\documentclass[a4paper]{amsart} 
\usepackage[T1]{fontenc}
\usepackage[margin=1in]{geometry}
\usepackage[american]{babel}
\usepackage[utf8x]{inputenc}
\usepackage[sc]{mathpazo}
\usepackage{amsmath,amssymb,amsfonts,mathrsfs}
\usepackage[colorlinks=true]{hyperref}
\usepackage{amsthm}
\usepackage{graphicx}
\usepackage{dutchcal}
\usepackage{varioref}
\usepackage{datetime}
\usepackage{cancel}
\usepackage{mathtools}

\usepackage{booktabs}

\usepackage{tikz-cd}
\usepackage{csquotes}
\usepackage{epigraph}

\usepackage{interval}
\intervalconfig{soft open fences}
\usepackage{dsfont}
\usepackage{microtype}

\usepackage{faktor}
\usepackage{xcolor}
\usepackage{prerex}
\usepackage{epsdice}
\usepackage{booktabs}
\usepackage[normalem]{ulem}
\usepackage{thmtools}
\usepackage{thm-restate}

\usepackage{pgfplots}
\pgfplotsset{compat=newest}
\usepgfplotslibrary{groupplots}
\usepgfplotslibrary{polar}
\usepgfplotslibrary{smithchart}
\usepgfplotslibrary{statistics}
\usepgfplotslibrary{dateplot}
\usepgfplotslibrary{ternary}
\usepgfplotslibrary{fillbetween}

\DeclareMathAlphabet{\mathdutchcal}{U}{dutchcal}{m}{n}
\SetMathAlphabet{\mathdutchcal}{bold}{U}{dutchcal}{b}{n}
\DeclareMathAlphabet{\mathdutchbcal}{U}{dutchcal}{b}{n}

\newtheorem{theorem}{Theorem}[section]
\newtheorem{proposition}[theorem]{Proposition}
\newtheorem{lemma}[theorem]{Lemma}
\newtheorem{corollary}[theorem]{Corollary}

\theoremstyle{definition}
\newtheorem{definition}[theorem]{Definition}

\theoremstyle{remark}
\newtheorem{remark}[theorem]{Remark}
\theoremstyle{conjecture}

\theoremstyle{question}

\numberwithin{equation}{section}

\newcommand{\C}{\mathbb{C}}

\newcommand{\N}{\mathbb{N}}

\newcommand{\R}{\mathbb{R}}
\newcommand{\Z}{\mathbb{Z}}

\newcommand{\T}{\mathbb{T}}

\DeclarePairedDelimiter\abs{\lvert}{\rvert}
\DeclarePairedDelimiter\norm{\lVert}{\rVert}

\renewcommand{\epsilon}{\ensuremath\varepsilon}

\renewcommand{\phi}{\ensuremath{\varphi}}

\newcommand{\eps}{\epsilon}

\newcommand{\anyfield}{\mathbb{F}}
\NewDocumentCommand\field{e{_}}{\IfValueF{#1}{\anyfield}\IfNoValueF{#1}{\Z_{#1}}}

\newcommand{\Reals}{\R}
\newcommand{\Nats}{\N}
\newcommand{\into}{\hookrightarrow}

\newcommand{\longto}{\longrightarrow}

\newcommand{\Laplace}{\Delta}
\newcommand{\E}{\mathbb{E}}

\newcommand{\Prob}{\mathbb{P}}

\newcommand{\muGP}{\mu_{\Ham}}

\newcommand{\reg}{\mathfrak{r}}
\newcommand{\scl}{\mathcal{scl}}
\newcommand{\cl}{\mathcal{cl}}
\newcommand{\ldd}{\text{law-defining datum}}
\newcommand{\ldds}{\text{law-defining data}}

\newcommand{\randomness}{\vcenter{\hbox{\scalebox{0.5}{\epsdice{5}}}}}
\newcommand{\rand}{\randomness}
\newcommand{\randH}{H^{\randomness}}

\newcommand{\ul}[1]{\underline{#1}}

\DeclareMathOperator{\id}{id}

\DeclareMathOperator{\vol}{vol}
\DeclareMathOperator{\supp}{supp}

\DeclareMathOperator{\Hof}{Hofer}
\DeclareMathOperator{\Ham}{Ham}

\DeclareMathOperator{\Symp}{Symp}
\DeclareMathOperator{\Diff}{Diff}

\DeclareMathOperator{\Aut}{Aut}

\DeclareMathOperator{\osc}{osc}

\DeclareMathOperator{\Cov}{Cov}

\DeclareMathOperator{\End}{End}
\DeclareMathOperator{\LDD}{LDD}
\DeclareMathOperator{\Sym}{Sym}
\DeclareMathOperator{\pr}{pr}

\definecolor{ceruleanblue}{rgb}{0.16, 0.32, 0.75}

\tikzset{    
    mypoint/.style={
        circle,
        draw,
        inner sep=.3mm
        },  
    whitepoint/.style={
        fill=white, 
        mypoint
        },  
    blackpoint/.style={
        fill=black, 
        mypoint
        },  
    textnode/.style={
        text height=2.5ex, 
        text depth=1ex
        },  
    }

\hypersetup{
linkcolor=blue,
citecolor=blue,
urlcolor=blue
}

\begin{document}

\date{\today}
\title[Random Hamiltonians II]{Random Hamiltonians II: A central limit theorem and the Hofer geometry of random walks}
\author{Adrian Dawid}
\address{Department of Pure Mathematics and Mathematical Statistics, University of Cambridge}
\email{apd55@cam.ac.uk}

\begin{abstract}
This paper investigates the global geometry of the group of Hamiltonian diffeomorphisms $\operatorname{Ham}(M,\omega)$ using random walks.
On a large class of symplectic manifolds, we show that the expected Hofer norm of such a random walk grows at least as fast as the square root of the number of steps.
Furthermore, we show that if the random walk is restricted to an abelian subgroup, then the growth rate is also bounded from above by the square root of the number of steps. 
This provides, for subgroups of the group of Hamiltonian diffeomorphisms, a probabilistic version of the flatness observed in commutative finite-dimensional Lie groups.
We also provide some numerical evidence suggesting that this upper bound fails away from commutative subgroups.
En route, we show that the class of probability measures introduced in~\cite{dawid-2025} is a class of Borel measures with respect to the $C^\infty$-topology on $\operatorname{Ham}(M,\omega)$, show the measurability of the stable commutator length, and show a central limit theorem for Hofer-Lipschitz quasimorphisms on $\operatorname{Ham}(M,\omega)$.
\end{abstract}
\maketitle

\tableofcontents
\section{Introduction}
This paper uses probabilistic methods to study the large-scale geometry of the group of Hamiltonian diffeomorphisms endowed with the Hofer metric.
Many aspects of the large-scale geometry of this group remain mysterious.
The behavior of random walks can shed light on the large-scale geometry of a metric space, and in particular on its curvature properties.
The prequel~\cite{dawid-2025} has introduced a family of probability measures on the group of Hamiltonian diffeomorphisms which 
are adapted to the Hofer metric.
By further developing this method, 
we study the asymptotic behavior of random walks on $\Ham(M,\omega)$.
We show that on a large class of symplectic manifolds, the expected Hofer distance of the random walk from the identity grows at least as fast as the square root of the number of steps.
We show the same asymptotic growth result on the universal cover of $\Ham(M,\omega)$ for any closed toric symplectic manifold.
In contrast to the prequel, we make use of features of ``hard'' symplectic topology, in particular spectral invariants, to obtain these results.
On the $2$-sphere, we also show that when restricted to an abelian subgroup, naturally arising from the toric structure, the expected Hofer distance of the random walk from the identity grows exactly as fast as the square root of the number of steps asymptotically.
En route, we show a new central limit theorem for Lipschitz quasimorphisms on metric groups with a right-invariant metric, which is of independent interest.
We also continue to develop the class of probability measures on $\Ham(M,\omega)$ introduced in~\cite{dawid-2025}, and
show that this class is closed under convolution and equivariant under the action of the group of symplectomorphisms by conjugation.
Additionally, we show that these measures depend continuously (with respect to the Wasserstein metric) on the data used to construct them.
We also significantly expand the class of measurable events by showing that the probability measures introduced in~\cite{dawid-2025} are Borel measures with respect to the $C^k$-topology for any $k \in \Nats_0 \cup \{\infty\}$. 

\subsection{Main results}
The central aim of this paper is to study the asymptotic behavior of random walks on the group of Hamiltonian diffeomorphisms
and gain insight into the large-scale geometry of this group.
Given that we have a random walk, one might ask about its expected distance from the identity with respect to the Hofer metric, i.e., about the growth of
$\E[d_{\Hof}(\id, \Phi_n)]$.
It is obvious that this can be at most linear, i.e., 
$\E[d_{\Hof}(\id, \Phi_n)] \leq Cn$ for some $C > 0$ depending on the \ldd.
The behavior of this expected distance should be interpreted as a probabilistic proxy for the large-scale geometry of $\Ham(M,\omega)$, in particular for the curvature of the group.
If $\E[d_{\Hof}(\id, \Phi_n)]$ stays bounded, this would indicate positively curved behavior, whereas if it grows linearly, this would indicate negatively curved behavior.
If it grows like $\sqrt{n}$, this would indicate flat behavior.

In this paper, we show that whenever a non-trivial Hofer-Lipschitz quasimorphism exists, then $\E[d_{\Hof}(\id, \Phi_n)]$ is indeed at least of order $\sqrt{n}$.
In particular, we show the following theorem:
\begin{restatable}{theorem}{lowerBoundOnHofer}
    \label{thm:lower-bound-Ham}
    Let $(M,\omega)$ be a closed symplectic manifold and $\mathcal{D}$ a centered autonomously exhaustive law-defining datum on $M$ such that there is some non-trivial Hofer-Lipschitz homogeneous quasimorphism $q: \Ham(M,\omega) \to \Reals$.
    Then there is some $c > 0$ such that $\E[d_{\Hof}(\id, \Phi_n)] \geq c\sqrt{n}$ for all $n$, where $\Phi_n$ is the random walk induced by $\mathcal{D}$. 
    Furthermore, almost all sample paths of $\Phi_n$ leave any bounded region of $\Ham(M,\omega)$ eventually.
\end{restatable}
Such quasimorphisms are known to exist on a large class of symplectic manifolds, for example on the sphere $S^2$, all complex projective spaces $\mathbb{C}P^n$, and many other toric symplectic manifolds, see~\cite{entov-polterovich-2003, usher-2011, fukaya-oh-ohta-ono-2019, cristofaro-gardiner-humiliere-mak-seyfaddini-smith-2022}.
Note that the construction of these quasimorphisms is based on Floer theory, and therefore is in particular a feature of ``hard'' symplectic topology.
We also prove the same lower bound on the expected distance of the random walk from the identity on the universal cover of $\Ham(M,\omega)$ for any closed toric symplectic manifold, see Corollary~\ref{cor:lower-bound-universal-cover}.
For finite-dimensional Lie groups there is a direct relationship between flatness and commutativity, see~\cite{milnor-1976}.
As a probabilistic analogue, we show that if the random walk is restricted to an abelian subgroup of $\Ham(M,\omega)$, then the expected distance is also bounded from above by a constant times $\sqrt{n}$.
\begin{restatable}{theorem}{upperBoundCommutative}
    \label{thm:upper-bound-commutative}
    Let $\mathcal{D}$ be a centered
    and commutative law-defining datum on a closed symplectic manifold $(M,\omega)$.
    Denote by $\Phi_k$ the associated random walk on $\Ham(M,\omega)$.
    Then there exists a constant $C > 0$ such that $\E[d_{\Hof}(\id, \Phi_k)] \leq C\sqrt{k}$ for all $k \in \Nats$.
\end{restatable}
On $S^2$, complex projective spaces, and their monotone products, we combine these two bounds to show that 
for a certain class of \ldds\ on $S^2$, the expected distance of the random walk from the identity grows exactly as fast as $\sqrt{n}$ (see Corollary~\ref{cor:exact-growth-S2} and Corollary~\ref{cor:exact-growth-CPn}).
The support of such an \ldd\ is contained in an abelian subgroup of the Hamiltonian diffeomorphism group, namely that of Hamiltonian diffeomorphisms induced by Hamiltonians that are invariant under the standard $S^1$-action on $S^2$, or the torus action on $\C P^n$ and products of complex projective spaces.
One big question remains: is $\sqrt{n}$ still the best asymptotic upper bound 
if we do not restrict the random walk to an abelian subgroup?
In the final section of this paper, we will discuss this question and possible directions for future investigation. In particular, we present some numerical evidence against the upper bound of $\sqrt{n}$ for the expected distance from the identity in the case of a random walk on the full group $\Ham(\T^2,\omega)$.
This is to say that despite the well-known prevalence of large quasi-flats in the Hamiltonian diffeomorphism group, the group might also exhibit some negatively curved behavior.
This potential hyperbolic behavior has so far not been detected by any deterministic method.

The key to the proof of Theorem~\ref{thm:lower-bound-Ham} is a central limit theorem for Lipschitz quasimorphisms on metric groups, which we establish in Section~\ref{sec:clt}.
Here we expand on a result of Björklund and Hartnick~\cite{bjorklund-hartnick-2011},
which establishes a central limit theorem for quasimorphisms on locally compact topological groups.
In lieu of the assumption of local compactness, we assume that the quasimorphism is Lipschitz with respect to a right-invariant metric on the group.
\begin{restatable}{theorem}{CLT}
    \label{thm:q-clt}
	Let $G$ be a metric group endowed with a right-invariant metric and let $\mu$ be a Borel probability measure on $G$.
	Let $q: G \to \Reals$ be a homogeneous Lipschitz quasimorphism that is square-integrable with respect to $\mu$.
	Further, let $(z_n)_{n \in \Nats}$ be a random walk on $G$ with $\mu$-distributed increments. 
	Then,
	\begin{equation}\label{eq:q-clt}
		\frac{q(z_n) - n \ell_\mu(q)}{\sqrt{n}} \xrightarrow{d} \mathcal{N}(0,\sigma^2), 
	\end{equation} 
	and $\sigma^2 > 0$ if and only if $q$ is not $\mu$-tame.
	Furthermore, if $\sigma^2 > 0$, 
	then 
	\begin{equation*} 
		\Prob \left[ \limsup_{n \to \infty} \frac{q(z_n) - n \ell_\mu(q)}{\sqrt{n \log \log n}} = \sigma \right] = 1.
	\end{equation*}
\end{restatable}
We refer the reader to Section~\ref{sec:prelims:quasimorphisms} for the definition of $\mu$-tameness and the $\mu$-distortion $\ell_\mu(q)$.
The above theorem implies that any homogeneous quasimorphism on $\Ham(M,\omega)$ that is Hofer-Lipschitz satisfies a central limit theorem when evaluated along the random walk $(\Phi_n)_{n \in \Nats}$.

Before proving these main results, we first establish some additional properties of the random Hamiltonian flow.
In particular, having a central limit theorem for quasimorphisms 
naturally suggests that we also investigate
the behavior of the stable commutator length along the random walk as these quantities are closely related by Bavard duality, see~\cite{bavard-1991}.
To show the measurability of the stable commutator length, we use recent results of Edtmair on the smooth perfectness of $\Ham(M,\omega)$, see~\cite{edtmair-2025}.
To make use of this fact, we first need to expand the realm of measurable events to include Borel sets with respect to the $C^\infty$-topology.
In~\cite{dawid-2025}, it is shown that the measures induced by \ldds\ are Borel measures with respect to the Hofer-topology on $\Ham(M,\omega)$.
In Section~\ref{sec:topologies}, we extend this result to show that the measures can also be considered as Borel measures with respect to the $C^k$-topology for any $k \in \Nats_0 \cup \{\infty\}$.
\begin{restatable}{theorem}{moreTopologies}
\label{thm:measurability_Ck}
    Let $\mathcal{D}$ be a \ldd, and let $\muGP^{\mathcal{D}}$ be the associated measure on $\Ham(M,\omega)$.
    Then, for any $k \in \Nats_0 \cup \{\infty\}$, the measure $\muGP^{\mathcal{D}}$ is a Borel measure with respect to the $C^k$-topology on $\Ham(M,\omega)$. 
\end{restatable}
Note that this result is of independent interest, as the $C^\infty$-topology is the natural topology on $\Ham(M,\omega)$ when viewing it as a subgroup of the diffeomorphism group $\Diff(M)$.
Using the above result, we show that the stable commutator length and any homogeneous quasimorphism are random variables with respect to the probability measures induced by \ldds, and we establish sub-Gaussian estimates for these random variables.
These results will be proven in Section~\ref{sec:scl-stats}. Leading up to this, we establish some 
general results on the family of probability measures induced by \ldds\ in Section~\ref{sec:ldds}.
In particular, we show that the measures depend continuously on the \ldd\ with respect to the Wasserstein metric, and that the class of measures is closed under convolution and equivariant under the action of the group of symplectomorphisms by conjugation.
\subsection{Acknowledgements}
First and foremost, I would like to thank my advisor Ivan Smith for his support.
During the writing of this paper I had several illuminating conversation about part one of this 
series of two papers, which have greatly influenced the contents of the present text. 
I particularly wish to thank 
Mohammed Abouzaid, 
Paul Biran, 
Erman Cineli,
Oliver Edtmair, 
Amanda Hirschi, 
Yongsheng Jia,
Francesco Morabito,
Baptiste Serraille, 
Sobhan Seyfaddini,
Egor Shelukhin, 
and 
Ibrahim Trifa for helpful discussions and comments on the first part of this series of two papers.
I would specifically like to thank Leonid Polterovich for suggesting to investigate 
the growth of the expected Hofer norm of the random walk.
I wrote part of this paper while visiting ETH Zurich during the 
fall semester 2025. I would like to thank Sobhan Seyfaddini and the whole symplectic geometry group at ETH for their warm hospitality throughout my stay.
This visit was partially supported by ERC Starting Grant 851701 "Homeomorphisms in symplectic topology and dynamics".
During the writing of this paper I was supported by EPSRC Horizon
Europe Guarantee grant “Floer theory beyond Floer (FloerPlus35)“ (project reference EP/X030660/1).
\section{Preliminaries}
In the following we will briefly recall some background material that will be needed later on.
A reader already familiar with the material might wish to skip this section.
Before we start with some more specific background material, 
let us fix some notation and conventions that will be used throughout the paper.
Henceforth, $(M,\omega)$ will always denote a closed symplectic manifold, and $\Ham(M,\omega)$ will denote the group of Hamiltonian diffeomorphisms of $(M,\omega)$.
For any smooth function $H \in C^\infty([0,1] \times M)$ the Hamiltonian vector field $X_H$ is defined by the equation $\omega(X_H, \cdot) = -dH$.
Furthermore, we denote the subset of autonomous Hamiltonian diffeomorphisms by $\Aut(M,\omega) \subset \Ham(M,\omega)$, i.e., $\Aut(M,\omega) = \{\phi_H^1 \mid H \in C^\infty(M)\}$.
When not otherwise specified, we will always assume that the Hamiltonian functions are normalized, i.e., $\int_M H(t,x) \omega^n = 0$ for all $t \in [0,1]$.
Usually, we will consider $\Ham(M,\omega)$ as a metric space with respect to the Hofer metric.
Recall that the Hofer norm of a Hamiltonian diffeomorphism $\phi \in \Ham(M,\omega)$ is defined by
\begin{equation*}
    \norm{\phi}_{\Hof} = \inf \left\{\int_0^1 \max_{x \in M} H(t,x) - \min_{x \in M} H(t,x) \, dt \mid \phi_H^1 = \phi\right\}.
\end{equation*}
The Hofer metric is then defined by $d_{\Hof}(\phi, \psi) = \norm{\phi^{-1} \circ \psi}_{\Hof}$ for $\phi, \psi \in \Ham(M,\omega)$.
Note that the Hofer metric is bi-invariant, i.e., $d_{\Hof}(\phi \circ \psi, \phi \circ \eta) = d_{\Hof}(\psi, \eta) = d_{\Hof}(\psi\circ\phi, \eta\circ \phi)$ for all $\phi, \psi, \eta \in \Ham(M,\omega)$. Furthermore, it can be lifted to a (pseudo-)metric on the universal cover $\widetilde{\Ham}(M,\omega)$ of $\Ham(M,\omega)$, which we will denote by $\widetilde{d}_{\Hof}$. In this case, the infimum is only taken over paths in the correct homotopy class of the given element in $\widetilde{\Ham}(M,\omega)$. 
\subsection{Law-defining data and the measure on $\Ham(M,\omega)$}\label{sec:prelims:ldds}
For the convenience of the reader, we will give a very brief summary of the construction from~\cite{dawid-2025}.
For all details we refer the interested reader to the arguments therein.

Let $(M,\omega)$ be a closed symplectic manifold.
The probability measures on $\Ham(M,\omega)$ that we will consider are constructed from a choice of \ldd\ $\mathcal{D}$, defined below.
\begin{definition}\label{def:ldd}
    A \textit{\ldd} $\mathcal{D}$ is a tuple $(\reg, J, \{e_n\}_{n \in \Nats_{>0}},  \{Z_n\}_{n \in \Nats_{>0}})$ such that 
    \begin{enumerate}
        \item $\reg > 0$, called the \textit{regularity parameter}, is a positive real number;
        \item $J$ is an $\omega$-compatible almost-complex structure on $M$;
        \item $\{\frac{1}{\vol_g(M)}, e_1, e_2, \ldots\}$ is an orthonormal basis of $L^2(M,g)$ and any $e_n$ is an eigenfunction of the Laplace-Beltrami operator $\Laplace_g: W^{2,2}(M, g) \to L^2(M, g)$ associated with the metric $g(\cdot,\cdot)=\omega(\cdot, J \cdot)$.
        Further, let $0 = \lambda_0 < \lambda_1 \leq \lambda_2 \leq \ldots < \infty$ be the eigenvalues of $\Laplace_g$ counted with multiplicities. Then $\{e_n\}_{n \in \Nats_{>0}}$ is ordered by increasing eigenvalue, i.e., $e_n$ is an eigenfunction with eigenvalue $\lambda_n$;
        \item $\{Z_n\}_{n \in \Nats_{>0}}$ is a family of Gaussian processes on $[0,1]$
        such that there exist constants $C,D \geq 0$ such that for all $n \in \Nats_{>0}$ and $t \in [0,1]$
        we have $\Cov[Z_n(t),Z_n(t)] < C$ and 
        $\E[\norm{Z_n}_\infty] < D$.
        We further require that $t_1, t_2 \mapsto \Cov[Z_n(t_1),Z_n(t_2)]$ and $t \mapsto \E[Z_n(t)]$ are smooth functions for any $n \in \Nats_{>0}$, which implies that the sample paths of $Z_n$ are almost-surely smooth.
        We say that the process $Z_n$ is \textit{associated with the eigenfunction $e_n$} in $\mathcal{D}$.
    \end{enumerate}
\end{definition}
In~\cite[Definition 3.2]{dawid-2025} several subtypes of law-defining data are defined, and we refer the interested reader to that definition for the details.
Given a law-defining datum $\mathcal{D}$, we can construct a probability measure 
on the space 
\begin{equation*}
    C_0^\infty([0,1] \times M) \coloneqq \{H \in C^\infty([0,1] \times M) \mid \int_M H(t,x) \omega^n = 0 \text{ for all } t \in [0,1]\}
\end{equation*}
of smooth normalized Hamiltonian functions on $M$.
The construction is as follows.
Let $\mathcal{D} = (\reg, J, \{e_n\}_{n \in \Nats_{>0}},  \{Z_n\}_{n \in \Nats_{>0}})$ be a law-defining datum. Then 
we define the random Hamiltonian function $\randH$ by
\begin{equation}\label{eq:GP_H}
    \randH(t,x) = \sum _{n \ge 1} w_n \cdot Z_n(t) \cdot e_n(x),
\end{equation} 
where $w_n = \exp(-\frac{1}{2}\lambda_n \cdot \reg)$.
We refer the reader to~\cite{dawid-2025} for the proof of 
the fact that the law of this random variable induces a probability measure $\muGP^{\mathcal{D}}$ on $C_0^\infty([0,1] \times M)$.
The measures on $\Ham(M,\omega)$ that we will consider are then obtained by pushing forward the
law of $\randH$ under the map $H \mapsto \phi_H^1$, where $\phi_H^1$ is the time-$1$ map of the Hamiltonian flow generated by $H$.
The measure on $\Ham(M,\omega)$ obtained in this way will be denoted by $\muGP^{\mathcal{D}}$.
We summarize the main existence result from~\cite{dawid-2025} in the following theorem.
\begin{theorem}\label{thm:existence-of-measure}
    Let $\mathcal{D}$ be any \ldd .
    Then the measure $\muGP^{\mathcal{D}}$ is a Hofer-Borel probability measure on $\Ham(M,\omega)$, i.e., $\muGP^{\mathcal{D}}(\Ham(M,\omega)) = 1$
    and any Hofer-open set $U \subset \Ham(M,\omega)$ is measurable.
    Furthermore, the measure $\muGP^{\mathcal{D}}$ has finite moments with respect to the Hofer norm, i.e., for any $p \in [1,\infty)$ we have
    \begin{equation*}\int_{\Ham(M,\omega)} \norm{\phi}_{\Hof}^p \, d\muGP^{\mathcal{D}}(\phi) < \infty.\end{equation*}
\end{theorem}
\subsection{The Wasserstein metric}\label{sec:prelims:wasserstein}
Let $(X,d)$ be a metric space.
We briefly recall the Wasserstein metric on the space of probability measures on $X$.
Note that we do not require $(X,d)$ to be Polish, which is sometimes assumed in the literature.
Let $\mathcal{P}(X)$ denote the space of probability measures on $X$,
and let $\mathcal{P}_p(X)$ denote the space of probability measures on $X$ with finite $p$-th moment, i.e.
\begin{equation*}
    \mathcal{P}_p(X) = \left\{\mu \in \mathcal{P}(X) \mid \int_X d(x_0,x)^p \, d\mu(x) < \infty \text{ for some (and hence all) } x_0 \in X\right\}.
\end{equation*}
Let $\mu, \nu \in \mathcal{P}_p(X)$.
Then we define the set of couplings of $\mu$ and $\nu$ by
\begin{equation*}
    \Pi(\mu,\nu) = \{\pi \in \mathcal{P}(X \times X) \mid (\pr_1)_* \pi = \mu, (\pr_2)_* \pi = \nu\},
\end{equation*}
where $\pr_1, \pr_2\colon X \times X \to X$ are the projections onto the first and second factor, respectively.
The idea of a coupling is that it is a joint distribution of two random variables with given marginals.
We can now define the $p$-Wasserstein distance between $\mu$ and $\nu$ by
\begin{equation*}
    W_p(\mu,\nu) = \inf_{\pi \in \Pi(\mu,\nu)} \left(\int_{X \times X} d(x,y)^p \,d\pi(x,y)\right)^{1/p}.
\end{equation*}
The space $\mathcal{P}_p(X)$ equipped with the $p$-Wasserstein distance is a metric space, and we denote it by $(\mathcal{P}_p(X), W_p)$.
Note that the original metric space $(X,d)$ can always be isometrically embedded into $(\mathcal{P}_p(X), W_p)$ by sending a point $x \in X$ to the Dirac measure $\delta_x \in \mathcal{P}_p(X)$.
Since there is a unique coupling of two Dirac measures, we have $W_p(\delta_x, \delta_y) = d(x,y)$ for all $x,y \in X$.
The Wasserstein metric is commonly used in the field of optimal transport, where it has had many applications in recent years.
We refer the interested reader to~\cite{villani-2009} for a more detailed introduction to the Wasserstein metric and its properties.
\subsection{Quasimorphisms}\label{sec:prelims:quasimorphisms}
In the following, we will briefly recall the definition of a quasimorphism and some of its properties.
For this, let $(G, \cdot)$ be a group.
\begin{definition}
    A function $q\colon G \to \R$ is called a \textit{quasimorphism} if there exists a constant $D \geq 0$ such that
    \begin{equation*}
        |q(g \cdot h) - q(g) - q(h)| \leq D \quad \text{for all } g,h \in G.
    \end{equation*}
    The smallest such constant $D$ is called the \textit{defect} of $q$ and is denoted by $D(q)$.
\end{definition}
The idea of a quasimorphism is that it is a function that is almost a homomorphism, up to a bounded error. In particular, a homomorphism is a quasimorphism with defect $0$.
However, any homomorphism $f\colon G \to \R$ 
necessarily has a kernel that is a normal abelian subgroup of $G$, which is a very strong restriction on the group $G$.
Even if there are no interesting homomorphisms from $G$ to $\R$, there can still be a large family of interesting quasimorphisms on $G$. 
In this paper, we will mostly restrict our attention to 
\emph{homogeneous} quasimorphisms, which are quasimorphisms that satisfy the additional property $q(g^n) = n \cdot q(g)$ for all $g \in G$ and $n \in \Z$.
Given any quasimorphism $q:G \to \R$, it can be homogenized to a homogeneous quasimorphism $\bar{q}$ by defining
\begin{equation*}
    \bar{q}(g) = \lim_{n \to \infty} \frac{q(g^n)}{n}.
\end{equation*}
Note that we call two quasimorphisms $q_1, q_2: G \to \R$ equivalent if $q_1 - q_2$ is bounded.
In any such equivalence class of quasimorphisms, there is a unique homogeneous quasimorphism, which can be obtained by homogenizing any quasimorphism in the equivalence class.
In the following, we will need the following definition specific to quasimorphisms on metric groups equipped with a Borel probability measure.
\begin{definition}\label{def:mu-tame}
	Let $G$ be a metric group endowed with a Borel probability measure $\mu$.
	Let $q$ be a $\mu$-integrable homogeneous quasimorphism on $G$. 
	We define the \emph{$\mu$-distortion} $\ell_\mu(q)$ of $q$ by 
	\begin{equation*}
	\ell_\mu(q) = \lim_n \frac{1}{n} \int_G q \, d\mu^{*n}.
	\end{equation*}
	We say that $q$ is \emph{$\mu$-tame} if there is a constant $C$ such that 
	$\abs{q(g) - n \ell_\mu(q) } \leq C$,
	for $\mu^{*n}$-almost every $g$ and all $n$. In particular, if $\mu$ is symmetric, then $\ell_\mu(q) = 0$ so that in this case $q$ is $\mu$-tame if
	and only if it is $\mu^{*n}$-essentially bounded on the subgroup closure of the support of $\mu$.
\end{definition}
\subsection{Quasimorphisms from Floer theory}\label{sec:prelims:floer}
In the following, we will briefly recall the construction of quasimorphisms on
$\Ham(M,\omega)$ from Floer theory, which were first introduced by Entov and Polterovich in~\cite{entov-polterovich-2003}.
Such quasimorphisms have since been studied extensively and found many applications in symplectic topology.
In this review, we will focus on some specific examples of quasimorphisms that will be relevant for our purposes.

\subsubsection{Using the Entov--Polterovich construction}
The first construction of quasimorphisms on $\Ham(M,\omega)$ from Floer theory was given by Entov and Polterovich in~\cite{entov-polterovich-2003}.
The main result is the following theorem, which we will use later on in this paper.
\begin{theorem}[Entov--Polterovich]\label{thm:entov-polterovich}
Let $(M,\omega)$ be one of the following symplectic manifolds:
\begin{enumerate}
\item the sphere $S^2$ with any area form $\omega$;
\item $S^2 \times S^2$ with the split symplectic form given by $\omega \oplus \omega$;
\item the complex projective space $\C P^n$ with the Fubini-Study form.
\end{enumerate}
Then there exists a non-trivial homogeneous 
quasimorphism $q: \Ham(M,\omega) \to \R$. Furthermore, $q$ is Hofer-Lipschitz, i.e., there exists a constant $C > 0$ such that
$|q(\phi)| \leq C \cdot \norm{\phi}_{\Hof}$ for all $\phi \in \Ham(M,\omega)$.
\end{theorem}
The idea behind the construction is to use \emph{spectral invariants} from Hamiltonian Floer theory 
which are associated with certain idempotents in the quantum homology of $M$.
Here one uses the fact that the quantum homology of $M$ is semi-simple, which is a property that holds for the manifolds in Theorem~\ref{thm:entov-polterovich}. Furthermore, the fact that quantum homology and Floer homology are related by the Piunikhin--Salamon--Schwarz isomorphism allows
associating spectral invariants to the idempotents in quantum homology.
A priori, the spectral invariants are only defined on the universal cover $\widetilde{\Ham}(M,\omega)$ of $\Ham(M,\omega)$, but in the cases considered in Theorem~\ref{thm:entov-polterovich} the quasimorphism descends to $\Ham(M,\omega)$.
This is due to the fact that the fundamental group $\pi_1(\Ham(M,\omega))$ is finite in these cases.
The overall strategy has been generalized to a larger class of symplectic manifolds.
The following theorem was proven independently 
by Usher (see~\cite[Theorem 1.6]{usher-2011}) and Fukaya--Oh--Ohta--Ono (see~\cite[Corollary 1.2]{fukaya-oh-ohta-ono-2019}):
\begin{theorem}[Usher, Fukaya--Oh--Ohta--Ono]\label{thm:toric-Entov-Polterovich} 
    Let $(M,\omega)$ be a closed toric symplectic manifold. Then there exists a homogeneous quasimorphism $q\colon \widetilde{\Ham}(M,\omega) \to \R$ which is Hofer-Lipschitz and non-trivial.
\end{theorem}
Again, the quasimorphism descends to $\Ham(M,\omega)$ if the fundamental group $\pi_1(\Ham(M,\omega))$ is finite
or more generally if the image of the Seidel homomorphism is finite.
Thus, the conclusion of Theorem~\ref{thm:entov-polterovich} can be extended to symplectic manifolds such as $((S^2)^n, \omega^{\oplus n})$ for any $n \in \Nats_{>0}$, or $\C P^{n_1} \times \ldots \times \C P^{n_k}$ with a monotone symplectic form, where $n_1, \ldots, n_k \in \Nats_{>0}$.
We refer the reader to~\cite{branson-2011} for a detailed discussion.
The general case of the theorem of Usher and Fukaya--Oh--Ohta--Ono is still of interest to us, as we will also study the universal cover of $\Ham(M,\omega)$ in this paper. 

\subsubsection{Using link spectral invariants}
On the sphere, there 
is another construction of quasimorphisms on $\Ham(S^2,\omega)$
due to 
Cristofaro-Gardiner, Humili\`ere, Mak, Seyfaddini and Smith.
We will give a brief outline of the properties that we will need later. 
The interested reader is referred to~\cite{cristofaro-gardiner-humiliere-mak-seyfaddini-smith-2022} for the details.
In the following, let $\Sigma$ be a closed surface with symplectic form $\omega$.
Let $\ul{L} = \cup_{i=1}^k L_i$ be a link, i.e., a disjoint union of embedded circles in $\Sigma$, such that the complement $\Sigma \setminus \ul{L} = \cup_{i=1}^s B_i^\circ$ is a disjoint union of planar domains $B_i^\circ$
and we assume that also the closures of the domains $B_i^\circ$ are planar.
We are interested in the following class of such links:
\begin{definition}\label{def:monotone-link}
Let $\ul{L}$ be a link as above. We call  $\ul{L}$ \emph{monotone} if there exists  $\eta \in \Reals_{\geq 0}$ such that 
\begin{equation} \label{eqn:independent_of_j}
2\eta(\tau_j-1)+A_j
\end{equation}
is independent of $j$, for $j \in \{1,\dots,s\}$.
Here $\tau_j$ is the number of boundary components of $B_j$ and $A_j$ is the area of $B_j$.
\end{definition}
In particular, if we have a link on the sphere $S^2$ such that the complement is a union of annuli and two disks with the same area, then the link is monotone.
Such a link can for example be obtained by taking two horizontal slices of the sphere.
Any given such slice can be included in a monotone link by adding the necessary number of additional slices, which are parallel to the first slice.
In fact, we will later only need this special case, which is pictured in Figure~\ref{fig:monotone-link}.
\begin{figure}[h]
    \centering
    \includegraphics[width=0.32\textwidth]{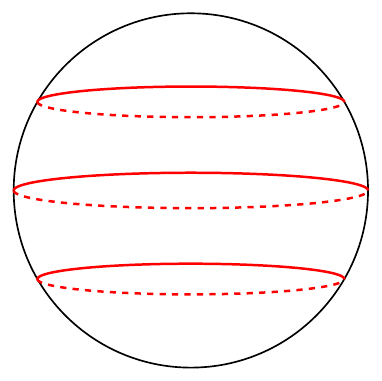}
    \hspace{1cm}
    \includegraphics[width=0.32\textwidth]{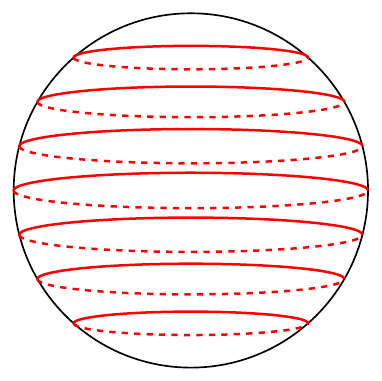}
    \caption{Two monotone links on the sphere $S^2$. On the left-hand side, the complement is a union of two disks and two annuli with the same area.
    On the right-hand side, it consists of two disks and six annuli with the same area.}\label{fig:monotone-link}
\end{figure}

Using Lagrangian Floer theory of the torus $\Sym^k(\ul{L})$ in the symmetric product $\Sym^k(\Sigma)$, the authors of~\cite{cristofaro-gardiner-humiliere-mak-seyfaddini-smith-2022} construct the following \emph{link spectral invariants}:
\begin{theorem}[Cristofaro-Gardiner, Humili\`ere, Mak, Seyfaddini and Smith]\label{thm:link-spectral-inv}
For every monotone Lagrangian link  $\ul{L}=\cup_{i=1}^k L_i$ there exists a  link spectral invariant \begin{equation*}c_{\ul{L}}\colon C^{\infty}([0,1]\times \Sigma,\omega) \to \mathbb{R}\end{equation*} satisfying the following properties.
\begin{itemize}
\item (Hofer Lipschitz) for any $H, H'$, we have
\begin{equation*}
\int_{0}^1 \min  (H_t - H'_t) dt \leq c_{\ul{L}}(H) - c_{\ul{L}}(H') \leq \int_{0}^1 \max\,  (H_t - H'_t) dt;
\end{equation*}
\item (Lagrangian control) if $H_t|_{L_i} = s_i(t)$ for each $i$, then 
\begin{equation*}
c_{\ul{L}}(H) = \frac{1}{k} \, \sum_{i=1}^k \int_0^1 s_i(t) dt;
\end{equation*}
\item (Homotopy invariance) if $H, H'$ are mean-normalized and determine the same point of the universal cover $\widetilde{\Ham}(\Sigma,\omega)$, then $c_{\ul{L}}(H) = c_{\ul{L}}(H')$;
\end{itemize}
\end{theorem}
We have compressed the statement of the theorem above to only include the properties that we will need later on.
Recall that if $\Sigma \not= S^2$, then $\Ham(\Sigma,\omega) = \widetilde{\Ham}(\Sigma,\omega)$.
Thus, by the homotopy invariance property of the link spectral invariant, 
the map descends to a map $c_{\ul{L}}: \Ham(\Sigma,\omega) \to \mathbb{R}$.
By the Hofer-Lipschitz property, the map $c_{\ul{L}}: \Ham(\Sigma,\omega) \to \mathbb{R}$ is Lipschitz with respect to the Hofer metric, and hence continuous.
However, this map is not a quasimorphism even after homogenization, 
as shown in~\cite[Proposition 17]{cheuk-yu-trifa-2026}.
If $\Sigma = S^2$, then $\Ham(S^2,\omega) \not= \widetilde{\Ham}(S^2,\omega)$, and the map $c_{\ul{L}}$ does not descend to a map on $\Ham(S^2,\omega)$ directly.
However, the authors of~\cite{cristofaro-gardiner-humiliere-mak-seyfaddini-smith-2022} use an averaging procedure to construct a quasimorphism on $\Ham(S^2,\omega)$ from the link spectral invariant $c_{\ul{L}}$.
In particular, for any $\varphi \in \Ham(S^2, \omega)$ we set
\begin{equation*}
\mu_{\ul{L}}(\varphi) \coloneqq \lim_{n \to \infty} \frac{c_{\ul{L}}(\tilde\varphi^n)}{n},
\end{equation*}
where $\tilde \varphi \in \widetilde{\Ham}(S^2, \omega)$ is any lift of $\varphi$.  
Interestingly, this limit is independent of the choice of lift $\tilde \varphi$ and defines a homogeneous quasimorphism $\mu_{\ul{L}}: \Ham(S^2, \omega) \to \R$.
\begin{theorem}[Cristofaro-Gardiner, Humili\`ere, Mak, Seyfaddini and Smith]\label{thm:quasimorphisms-links}
For any monotone Lagrangian link $\ul{L}$ on $S^2$, the map $\mu_{\ul{L}} \colon \Ham(S^2, \omega) \to \R$
as defined above
is a homogeneous quasimorphism.
Furthermore, it satisfies the following properties:
\begin{enumerate}
\item (Hofer Lipschitz)  $\abs{\mu_{\ul{L}}(\varphi) - \mu_{\ul{L}}(\psi)} \leq d_{\Hof}(\varphi, \psi)$;
\item (Lagrangian control) Suppose $H$ is mean-normalized.  If $H_t|_{L_i} = s_i(t)$ for each $i$, then 
\begin{equation*}
\mu_{\ul{L}}(H) = \frac{1}{k} \, \sum_{i=1}^k \int_0^1 s_i(t) dt.
\end{equation*}
\end{enumerate}
\end{theorem}
It should be noted that there are other constructions of Hofer-Lipschitz quasimorphisms on $\Ham(S^2, \omega)$, for example the original construction of Entov and Polterovich~\cite{entov-polterovich-2003}, which uses spectral invariants from Hamiltonian Floer theory.
For our purposes, the quasimorphisms constructed in~\cite{cristofaro-gardiner-humiliere-mak-seyfaddini-smith-2022} are the most convenient, and we refer the interested reader to the references above for more details on the other constructions.
\subsection{The $L^p$-metrics and Gambaudo--Ghys quasimorphisms}\label{sec:prelims:lp}
While the Hofer metric is the most commonly used metric on $\Ham(M,\omega)$, there are also other metrics that can be defined on $\Ham(M,\omega)$.
A well-studied family of such metrics are the $L^p$-metrics, which are defined as follows.
Let $J \colon TM \to TM$ be an $\omega$-compatible almost complex structure, and let $g(\cdot,\cdot) = \omega(\cdot, J \cdot)$ be the associated Riemannian metric on $M$.
Then for any smooth function $H \in C^\infty([0,1] \times M)$ we define the $L^p$-length of the Hamiltonian path $\phi_H^t$ by
\begin{equation*}
    \ell_p(\{\phi_H^t\}_{t \in [0,1]}) = \int_0^1 \left(\int_M \frac{1}{\vol(M,\omega)} \norm{X_H(t,\cdot)}_g^p \, \omega^n\right)^{1/p} dt.
\end{equation*}
The $L^p$-metric on $\Ham(M,\omega)$ is then defined by
\begin{equation*}
    d_p(\phi, \psi) = \inf\{\ell_p(\{\phi_H^t\}_{t \in [0,1]}) \mid \phi_H^1 = \psi \circ \phi^{-1}\}.
\end{equation*}
Note that since $d_p(\phi, \psi) = d_p(\id, \psi \circ \phi^{-1})$, we have $d_p(\phi \tau, \psi \tau) = d_p(\phi, \psi)$ for all $\phi, \psi, \tau \in \Ham(M,\omega)$, i.e., the $L^p$-metric is right-invariant.
The interest in the $L^p$-metrics stems from the fact that the $p=2$ case is related to the hydrodynamics of an ideal fluid, see~\cite{shnirelman-1994},
whilst the $L^1$-metric gives the average length of a trajectory of a Hamiltonian path.
The $L^p$-metrics are independent of the choice of $\omega$-compatible almost complex structure $J$ up to bi-Lipschitz equivalence.

There are certain quasimorphisms on $\Ham(M,\omega)$ that are Lipschitz with respect to the $L^p$-metrics.
In particular, in~\cite{brandenbursky-marcinkowski-shelukhin-2022} the authors show that the quasimorphisms constructed by Gambaudo-Ghys are Lipschitz with respect to the $L^p$-metric for any $p \geq 1$.
The Gambaudo-Ghys quasimorphisms are
defined on the group of area-preserving diffeomorphisms of a surface, and are constructed using braid invariants.
Intuitively, they encode a kind of average braiding of the trajectories of points under the Hamiltonian flow.
We will not go into the details of the construction here, as we only need their existence, non-triviality, and Lipschitz property with respect to the $L^p$-metrics. We refer the reader to~\cite{gambaudo-ghys-2004, brandenbursky-marcinkowski-shelukhin-2022} for more details.
The following combined result summarizes the properties of the Gambaudo-Ghys quasimorphisms that we will need later on in this paper.
\begin{theorem}[Gambaudo--Ghys, Brandenbursky--Marcinkowski--Shelukhin]\label{thm:gambaudo-ghys}
    Let $(\Sigma,\omega)$ be a closed surface with symplectic form $\omega$.
    Then there exists a family of non-trivial homogeneous quasimorphisms $\mu\colon \Ham(\Sigma,\omega) \to \R$ that are Lipschitz with respect to the $L^p$-metric for any $p \geq 1$.
\end{theorem}
\section{The family of law-defining data and the resulting measures}\label{sec:ldds}
The focus of~\cite{dawid-2025} was on the construction of a measure $\muGP^{\mathcal{D}}$ on $\Ham(M,\omega)$ from a single law-defining datum $\mathcal{D}$.
However, the family of all law-defining data has some useful properties, which we will now discuss.
For this, we define a space of law-defining data $\LDD(M,\omega)$, which we will see can be endowed with a natural topology.
For set-theoretic reasons, we will assume that all law-defining data are defined on the same probability space $(\Omega,\mathcal{F},\Prob)$.\footnote{Otherwise, there will be no set of all law-defining data, since the probability space is part of the definition of a law-defining datum and there is no set of all probability spaces.}
However, this is no restriction, since we could have taken $(\Omega,\mathcal{F},\Prob)$ to be any fixed standard probability space to begin with.
With this in mind, we can define the set of all law-defining data as
\begin{equation*}
    \LDD(M,\omega) \coloneqq \{ \mathcal{D} = (\reg,J,\{e_n\}_{n \ge 1},\{Z_n\}_{n \ge 1}) \mid \mathcal{D} \text{ is a law-defining datum} \}.
\end{equation*}
\subsection{Closedness under convolution}
Let us fix an $\omega$-compatible almost-complex structure $J$ on $(M,\omega)$.
Let $\mathcal{D}_1$ and $\mathcal{D}_2$ be two law-defining data with the same almost-complex structure $J$.
Assume that $\phi^{\rand},\psi^{\rand} \colon \Omega \to \Ham(M,\omega)$ are two independent random variables with law $\muGP^{\mathcal{D}_1}$ and $\muGP^{\mathcal{D}_2}$, respectively.
Then, we can ask: What is the law of the products $\phi^{\rand} \circ \psi^{\rand}$ and $\psi^{\rand}\circ \phi^{\rand}$? 
In the following, we will show that there exists a law-defining datum $\mathcal{D}_1 * \mathcal{D}_2$ such that the law of $\phi^{\rand} \circ \psi^{\rand}$ is $\muGP^{\mathcal{D}_1 * \mathcal{D}_2}$.

To formulate this without the unnecessary introduction of proxy random variables, we use the \emph{convolution} of measures. 
\begin{definition}
    Let $G$ be a topological group and let $\mu_1$ and $\mu_2$ be two Borel probability measures on $G$. Then the \emph{convolution} of $\mu_1$ and $\mu_2$ is the Borel probability measure $\mu_1 * \mu_2$ on $G$ defined by
    \begin{equation*}
        (\mu_1 * \mu_2)(A) \coloneqq \int_{G\times G} \mathbb{1}_{A}(gh) d\mu_1(g)d\mu_2(h)
    \end{equation*}
    for all Borel sets $A \subseteq G$.
\end{definition}
\begin{remark}
    Note that the convolution of two measures is associative, but not necessarily commutative if $G$ is non-abelian.
    Consider for example the group $G = \langle a,b \mid a^3 = b^3 = e\rangle$ and the measures $\mu_1 = \delta_{a}$ and $\mu_2 = \delta_{b}$. Then $\mu_1 * \mu_2 = \delta_{ab} \neq \delta_{ba} = \mu_2 * \mu_1$.
\end{remark}
We can now state the main result of this section, which is 
another way in which the family of all measures $\muGP^{\mathcal{D}}$ resembles 
the family of Gaussian measures on a vector space, which is closed under convolution.
\begin{theorem}\label{thm:convolution}
    Let $\mathcal{D}_1,\mathcal{D}_2 \in \LDD(M,\omega)$ be two law-defining data with the same almost-complex structure $J$ and the same regularity $\reg > 0$. Then there exists a law-defining datum $\mathcal{D}_1 * \mathcal{D}_2$ such that 
    \begin{equation*}
        \muGP^{\mathcal{D}_1 * \mathcal{D}_2} = \muGP^{\mathcal{D}_1} * \muGP^{\mathcal{D}_2}.
    \end{equation*}
\end{theorem}
\begin{proof}
    We use the following notation for the components of our two \ldds :
    \begin{equation*}
        \mathcal{D}_1 = (\reg,J,\{e^{(1)}_n\},\{Z^{(1)}_n\}) \qquad \text{and} \qquad
        \mathcal{D}_2 = (\reg,J,\{e^{(2)}_n\},\{Z^{(2)}_n\}).
    \end{equation*}
    Now, by construction, the convolution $\muGP^{\mathcal{D}_1} * \muGP^{\mathcal{D}_2}$ is the law of the random variable $\phi^1_{H^{(1)}} \circ \phi^1_{H^{(2)}}$,
    where $H^{(i)}$ is the random Hamiltonian function defined by $\mathcal{D}_i$ for $i = 1,2$, i.e.\
    \begin{equation*}
        H^{(i)}(t,x) = \sum _{n \ge 1} w_n \cdot Z_n^{(i)}(t) \cdot e^{(i)}_n(x). 
    \end{equation*}
    Since $\mathcal{D}_1,\mathcal{D}_2$ share the same almost-complex structure, 
    $\{e^{(1)}_n\}$ and $\{e^{(2)}_n\}$ can only differ in two ways: (1) by signs, or (2) by reordering within the (finite-dimensional)
    eigenspaces.
    Without loss of generality, we can relabel these bases and their associated coefficient processes as to agree in their ordering.
    To correct for signs, we introduce $\{\alpha_n\}_{n\in \Nats} \subset \{\pm1\}$ such that 
    \begin{equation*}
        H^{(2)}(t,x) = \sum _{n \ge 1} \alpha_n \cdot w_n \cdot Z_n^{(2)}(t) \cdot e^{(1)}_n(x). 
    \end{equation*}
    Let $\beta\colon \Reals \to \Reals_{\geq 0}$
    be a smooth bump function such that $\supp \beta \Subset (0,1)$ and $\int_0^1 \beta(t) dt = 1$.
    Then,
    the Hamiltonian function
    \begin{equation*}
        H^{(1)} 
        \sqcup H^{(2)}(t,x) \coloneqq 
        2 \cdot \left( \beta(2t) \cdot H^{(2)}(t,x) + \beta(2t-1) \cdot H^{(1)}(t,x) \right),
    \end{equation*}
    satisfies $\phi^1_{H^{(1)} \sqcup H^{(2)}} = \phi^1_{H^{(1)}} \circ \phi^1_{H^{(2)}}$.
    We are now almost done.
    Set
    \begin{equation*}
        \tilde{Z}_k(t) = \alpha_k \cdot \beta(2t) \cdot Z^{(2)}(t) + \beta(2t-1) \cdot Z^{(1)}(t).
    \end{equation*}
    For $k \ge 1$ we obtain a sequence of Gaussian processes $\tilde{Z}_k$ on $[0,1]$.
    Let $\mathcal{D}_1 * \mathcal{D}_2$ be the \ldd\ defined by replacing the coefficient processes of $\mathcal{D}$ 
    with $\tilde{Z}_k$ for $k \ge 1$.
    Note that the associated random Hamiltonian function of $\mathcal{D}_1 * \mathcal{D}_2$ is $H^{(1)} \sqcup H^{(2)}$.
    By construction, $\muGP^{\mathcal{D}_1 * \mathcal{D}_2}$ is the law of $\phi^1_{H^{(1)} \sqcup H^{(2)}}$, which is equal to $\muGP^{\mathcal{D}_1} * \muGP^{\mathcal{D}_2}$.
\end{proof}
\subsection{Equivariance under the $\Symp(M,\omega)$-action}
Note that the group of all symplectomorphisms $\Symp(M,\omega)$ acts on the subgroup of Hamiltonian diffeomorphisms $\Ham(M,\omega)$ by conjugation, i.e., for any $\phi \in \Symp(M,\omega)$ and any $\psi \in \Ham(M,\omega)$ we have that $\phi \circ \psi \circ \phi^{-1} \in \Ham(M,\omega)$.
This action allows us to define a push-forward of measures on $\Ham(M,\omega)$ by symplectomorphisms.
We denote this measure by $\phi_* \mu$ for any Borel measure $\mu$ on $\Ham(M,\omega)$ and any $\phi \in \Symp(M,\omega)$.
It is defined by $(\phi_* \mu)(A) \coloneqq \mu(\phi^{-1} A \phi)$ for all Borel sets $A \subseteq \Ham(M,\omega)$.

We will now define an action of $\Symp(M,\omega)$ on the space of law-defining data, which is compatible with the action of $\Symp(M,\omega)$ on $\Ham(M,\omega)$.
We start with the following observation: 
the measures $\muGP^{\mathcal{D}}$ 
depend on the choice of law-defining datum $\mathcal{D}$ and
 in particular on the choice of almost-complex structure $J$.
It is well-known that the space of $\omega$-compatible almost-complex structures on $(M,\omega)$ is contractible, see~\cite[Proposition 4.1.1]{mcduff-salamon-2017}.
However, while topologically trivial, the space of $\omega$-compatible almost-complex structures is vast.
We should note that there is a natural action of $\Symp(M,\omega)$ on the space of $\omega$-compatible almost-complex structures.
Let 
\begin{equation*}
    \mathcal{J}(M,\omega) \coloneqq \{ J \in \End(TM) \mid J^2 = -\id, \omega(\cdot,J\cdot) > 0\}
\end{equation*}
be the space of $\omega$-compatible almost-complex structures on $(M,\omega)$.
Then, for any $\phi \in \Symp(M,\omega)$ and any $J \in \mathcal{J}(M,\omega)$, we can define the push-forward almost-complex structure $\phi_*J$ by $(\phi_*J)_x \coloneqq d\phi_{\phi^{-1}(x)} \circ J_{\phi^{-1}(x)} \circ d\phi^{-1}_x$ for all $x \in M$.
Note that $\phi_*J$ is again an $\omega$-compatible almost-complex structure on $(M,\omega)$,
and in fact the map $\phi\colon M \to M$ is an isometry between the Riemannian metric $\omega(\cdot,J\cdot)$ and the Riemannian metric $\omega(\cdot,\phi_*J\cdot)$ as can be obtained by a straightforward computation.
Of course, $\phi^{-1}$ is an isometry in the other direction.
Thus, we also have that $\Laplace_J (f \circ \phi) = (\Laplace_{\phi_*J} f) \circ \phi$ and $\Laplace_{\phi_*J}(f \circ \phi^{-1}) = (\Laplace_J f) \circ \phi^{-1}$ for all smooth functions $f\colon M \to \Reals$.  
In particular, if $\{e_n\}_{n \ge 1}$ is an orthonormal basis of eigenfunctions of $\Laplace_J$, then $\{e_n \circ \phi^{-1}\}_{n \ge 1}$ is an orthonormal basis of eigenfunctions of $\Laplace_{\phi_*J}$ with the same eigenvalues.
Then, for any $\psi \in \Symp(M,\omega)$ and any $\mathcal{D} = (\reg,J,\{e_n\}_{n \ge 1},\{Z_n\}_{n \ge 1}) \in \LDD(M,\omega)$, we can define the push-forward law-defining datum $\psi_*\mathcal{D}$ by
\begin{equation*}
    \psi_*\mathcal{D} \coloneqq (\reg,\psi_*J,\{e_n \circ \psi^{-1}\}_{n \ge 1},\{Z_n\}_{n \ge 1}).
\end{equation*}
\begin{theorem}\label{thm:equivariance}
    Let $\mathcal{D}$ be a law-defining datum.
    Then for any $\psi \in \Symp(M,\omega)$ we have that
    \begin{equation*}
        \psi_* \muGP^{\mathcal{D}} = \muGP^{\psi_*\mathcal{D}}.
    \end{equation*}
\end{theorem}
\begin{proof}
    Note that the random Hamiltonian function $H^{\rand}$ associated with $\psi_*\mathcal{D}$ is given by
    \begin{equation*}
        H^{\rand}_{\psi_*\mathcal{D}}(t,x) = \sum_{n \ge 1} w_n \cdot Z_n(t) \cdot e_n(\psi^{-1}(x)).
    \end{equation*}
    Similarly, the random Hamiltonian function $H^{\rand}$ associated with $\mathcal{D}$ is given by
    \begin{equation*}
        H^{\rand}_{\mathcal{D}}(t,x) = \sum_{n \ge 1} w_n \cdot Z_n(t) \cdot e_n(x).
    \end{equation*}
    Clearly, we have that $H^{\rand}_{\psi_*\mathcal{D}}$ has the same law as $H^{\rand}_{\mathcal{D}} \circ \psi^{-1}$.
    Note that for any Hamiltonian function $H\colon [0,1] \times M \to \Reals$ and any $\psi \in \Symp(M,\omega)$ we have that $\phi^1_{H \circ \psi^{-1}} = \psi \circ \phi^1_H \circ \psi^{-1}$.
    Thus, for any Borel set $A \subseteq \Ham(M,\omega)$ we have that
    \begin{align*}
        \muGP^{\psi_*\mathcal{D}}(A)
        &= \Prob\left[\phi^1_{H^{\rand}_{\psi_*\mathcal{D}}} \in A\right] 
        = \Prob\left[\phi^1_{H^{\rand}_{\mathcal{D}} \circ \psi^{-1}} \in A\right] 
        = \Prob\left[\psi \circ \phi^1_{H^{\rand}_{\mathcal{D}}} \circ \psi^{-1} \in A\right] 
        = \Prob\left[\phi^1_{H^{\rand}_{\mathcal{D}}} \in \psi^{-1} A \psi\right] \\
        &= \muGP^{\mathcal{D}}(\psi^{-1} A \psi) 
        = (\psi_* \muGP^{\mathcal{D}})(A).
    \end{align*}
\end{proof}
\subsection{Continuous dependence on the law-defining datum}
In~\cite{dawid-2025} we have seen that 
the properties of the measure on $\Ham(M,\omega)$
are highly dependent on the \ldd\ used to define it.
For example, for some \ldds\ the measure has full support, while for others it is supported on $\Aut(M)$, or even on a single point.
In this section, we want to establish that
these measures depend continuously on the law-defining datum.

In order to make sense of that statement, we need to specify a topology 
on the space of \ldds\ and on the space of Borel measures on $\Ham(M,\omega)$.
We start with the latter.
Let $\mathcal{P}(\Ham(M,\omega), d_{\Hof})$ be the space of Borel probability measures 
on $\Ham(M,\omega)$ where this group is topologized with the Hofer metric.
Note that by~\cite[Theorem 1.2]{dawid-2025} and~\cite[Theorem 1.4]{dawid-2025}
all measures $\muGP^{\mathcal{D}}$ belong to the smaller space
$\mathcal{P}_p(\Ham(M,\omega), d_{\Hof})$ of measures with finite $p$-th moment.
We can now topologize the space $\mathcal{P}_p$ with the $p$-Wasserstein metric, see Section~\ref{sec:prelims:wasserstein} for the definition.

It remains to specify a topology on the space of law-defining data.
Since any \ldd\ is a tuple consisting of four elements, 
the most straightforward approach is to topologize these four elements and then consider the 
product topology.
The first is the regularity parameter, and we of course use the standard topology on $\Reals$.
On the space of $\omega$-compatible almost-complex structures $\mathcal{J}(M,\omega)$
we use the topology induced by the $C^2$-metric. The reason for this specific choice of 
regularity is that it is the minimal required regularity for the argument below.
To see why,
let us first recall some useful facts from Riemannian geometry.
Let $J \in \mathcal{J}(M,\omega)$ be an $\omega$-compatible almost complex structure on 
$(M,\omega)$.
Then we define the corresponding Riemannian metric as $g_J(\cdot,\cdot) = \omega(\cdot, J \cdot)$.
Note that the map $J \mapsto g_J$ is continuous in the $C^k$
topology for any $k > 0$.
In particular, this is true for the $C^2$-topology we will need later.

A law-defining datum keeps track of an almost-complex structure together 
with a preferred orthonormal basis of $L^2(M,1/(n!) \omega^n)$
made up of eigenfunctions of the Laplace-Beltrami operator of the associated metric.
We would thus like to understand how a neighborhood of these might look.

For this, we wish to understand how quickly the eigenfunctions of the Laplace-Beltrami operator can change when we change the almost-complex structure.
For this purpose, we 
use two tools from Riemannian geometry: 
the first is the following theorem of Donnelly, see~\cite[Theorem 1.6]{donnelly-2001}:
\begin{theorem}[Donnelly 2001]\label{thm:donnelly}
Let $M$ be a compact Riemannian manifold of dimension $n$ and $\Laplace$ the Laplace-Beltrami operator. Suppose that the injectivity radius of $M$ is bounded below by $c_2$ and that the absolute value of the sectional curvature of $M$ is bounded above by $c_3$. If $\Delta \phi = \lambda \phi$ and $\lambda \neq 0$, then
\begin{equation*}
\|\phi\|_\infty \leq c_1 \lambda^{\frac{n-1}{4}} \|\phi\|_2.
\end{equation*}
The constant $c_1$ depends only upon $c_2$, $c_3$ and the dimension of $M$.
\end{theorem}

It is easy to see, based on the definition, that the injectivity radius and the sectional curvature of a Riemannian manifold depend continuously on the Riemannian metric in the $C^2$-topology.
Thus, we can now establish a rigid control on the $\infty$-norm of the eigenfunctions of the Laplace-Beltrami operator associated with a Riemannian metric, in terms of the eigenvalues.
Thus, it remains to control their growth. 
The Weyl law tells us the asymptotic growth of the eigenvalues of the Laplace-Beltrami operator.
Note that since the volume form does not depend on the almost-complex structure, the Weyl law is independent of the almost-complex structure.
However, \textit{a priori} the quality of the estimate at any finite time depends on the Riemannian metric in an uncontrolled way.
Thus, we will instead use the following cruder but uniform estimate on the growth of the eigenvalues of the Laplace-Beltrami operator, which is a theorem of Cheng, see~\cite[Corollary 2.3]{cheng-1975}:
\begin{theorem}[Cheng 1975]\label{thm:cheng}
Suppose $M$ is a compact $n$-dimensional Riemannian manifold with Ricci curvature $\geq (n-1)(-k)$ for $k > 0$. Then, when $n = 2(m+1)$, $m \in \Nats_0$,
\begin{equation*}
\lambda_l(M) \leq \frac{(2m+1)^2}{4} k + \frac{4 l^2 (1 + 2^m)^2 \pi^2}{d_M^2},
\end{equation*}
and when $n = 2m+3$, $m \in \Nats_0$,
\begin{equation*}
\lambda_l(M) \leq \frac{(2m+2)^2 k}{4} + \frac{4 l^2 (1 + \pi^2)(1 + 2^{2m})^2}{d_M^2},
\end{equation*}
for all $l \geq 1$,
where $d_M$ is the diameter of $M$ and $\lambda_l(M)$ is the $l$-th eigenvalue of the Laplace-Beltrami operator on $M$.
\end{theorem}
These theorems tell us that if we consider only such orthonormal bases of $L^2(M, \frac{1}{n!} \omega^n)$
as arise from almost-complex structures lying in some $C^2$-ball,
then we can control their growth in the supremum norm by a uniform polynomial.

We note the following consequence: take $(C^\infty(M), d_{C^0})^{\times \N}$ 
with the product topology.
An open set in this topology is a set where only finitely many elements of the sequence are restrained to lie in some open 
sets in $C^\infty(M)$ with respect to the $C^0$-topology.
A priori, it might seem like this topology will contain too few open sets for our purpose.
However, if we now restrict to a subset of functions which are orthonormal eigenbases of $L^2(M, \frac{1}{n!} \omega^n)$
for some $g_J$, where $J$ itself also is constrained to lie in a $C^2$-ball, the situation is more amenable.
The bound given by the two theorems from Riemannian geometry allows us to constrain all elements of the sequence sufficiently.
Thus, we wish to topologize the space of sequences of eigenfunctions with the product topology, i.e., as  $(C^\infty(M), d_{C^0})^{\times \N}$.

We have now considered the almost-complex structure and the eigenbasis.
It remains to ascertain how we wish to treat the coefficient processes.
Recall that the coefficient processes are smooth stochastic processes on $[0,1]$.
A common type of convergence to consider for stochastic processes is uniform convergence in probability. 
A sequence $\{W_n\}$ of stochastic processes on the interval $[0,1]$ converges uniformly in probability
to a stochastic process $W$ 
if $\E[\norm{W_n - W}] \to 0$ as $n \to \infty$.
In our case, we do not want the coefficient processes to be considered independently.
Thus, we want uniformity both over the interval $[0,1]$ and the index.

With all these choices, it is easy to write down a metric inducing the correct topology.
Let $\mathcal{D} = (\reg, J, \{e_n\}, \{Z_n\})$ and $\mathcal{D}' = (\reg', J', \{e'_n\}, \{Z'_n\})$ be two law-defining data.
Then we set 
\begin{equation}\label{eq:metric-on-ldd}
    d(\mathcal{D},\mathcal{D}') \coloneqq |\reg - \reg'| + d_{C^2}(J,J') + \sum_{k \ge 1} \frac{\norm{e_k - e'_k}_{\infty}}{2^k} +  \E\left[\sup_{n,t} |Z_n(t) - Z'_n(t)|\right].
\end{equation}
From now on, $\LDD(M,\omega)$ is always considered as topologized with this metric.
As we have argued above, it induces a natural topology on $\LDD(M,\omega)$.
With the topology now established, 
we can now prove the following theorem, which is the main result of this section.
\begin{theorem}
\label{thm:wasserstein-continuity}
    Let $p \in [1,\infty)$ be arbitrary.
    Then the map
    \begin{align*}
        (\LDD(M,\omega),d) &\longto (\mathcal{P}_p(\Ham(M,\omega)), W_p) \\
        \mathcal{D} &\longmapsto \muGP^{\mathcal{D}}
    \end{align*}
    is continuous, where we metrize the space of \ldds\ with the metric defined in~\eqref{eq:metric-on-ldd}
    and the space $\mathcal{P}_p(\Ham(M,\omega))$ of probability measures on $\Ham(M,\omega)$ for which the Hofer norm has finite 
    $p$-th moment is endowed with the $p$-Wasserstein metric $W_p$.
\end{theorem}
\begin{proof}
    Let $(\reg, J, \{e_n\}, \{Z_n\}) = \mathcal{D} \in \LDD(M,\omega)$ be a law-defining datum 
    and $\eps>0$ be arbitrary.
    To any \ldd\ $\mathcal{D} \in \LDD(M,\omega)$ we
    associate the random Hamiltonian function $H^{\rand}_{\mathcal{D}}$ given by 
    \begin{equation*}
        H^{\rand}_{\mathcal{D}}(t,x) = \sum_{n \ge 1} w_n \cdot Z_n(t) \cdot e_n(x),
    \end{equation*}
    where $w_n = \exp(-1/2 \cdot \lambda_n \cdot \reg)$ and $\lambda_n$ is the $n$-th eigenvalue of the Laplace operator associated with the almost-complex structure of $\mathcal{D}$.
    Our chosen \ldd\ $\mathcal{D}$ comes with the almost-complex structure $J$, and we denote the associated Riemannian metric by $g_J$.
    Let $\iota > 0$ be the injectivity radius of $g_J$.
    Then there exists a $\delta > 0$ such that the injectivity radius of $g_{J'}$
    lies in $(\frac{1}{2}\iota,\frac{3}{2}\iota)$
    for all $J' \in \mathcal{J}(M,\omega)$ with $d_{C^2}(J,J') < \delta$.
    Note that the Ricci and sectional curvature in this ball are bounded.
    Thus, by Theorem~\ref{thm:donnelly}, 
    there exists a constant $C > 0$
    such that 
    \begin{equation*}
        \norm{e'_k}_\infty \leq C \cdot (\lambda_k(g_{J'}))^{\frac{\dim M -1}{4}}
    \end{equation*}
    holds for any eigenfunction $e'_k$ with eigenvalue $\lambda_k(g_{J'})$ of $g_{J'}$ with $d_{C^2}(J,J') < \delta$.
    Additionally, by Theorem~\ref{thm:cheng}, there exists a constant $C' > 0$ such that 
    \begin{equation*}
        \lambda_k(g_{J'}) \leq C'k^2
    \end{equation*}
    for any $k \geq 1$ and $J'$ with $d_{C^2}(J,J') < \delta$.
    By combining these two estimates, we 
    obtain that there is a constant $C'' > 0$ such that 
    \begin{equation*}
        \norm{e'_k}_\infty \leq C'' \cdot k^{\frac{\dim M -1}{2}}
    \end{equation*}
    whenever $d_{C^2}(J,J')<\delta$ and for any eigenfunction $e_k$ of the $k$-th eigenvalue
    of $\Laplace_{J'}$.

    Let us now consider $(\reg', J', \{e'_n\}, \{Z'_n\}) = \mathcal{D}' \in \LDD(M,\omega)$ with $d(\mathcal{D}, \mathcal{D}') < \delta$.
    Clearly, by~\eqref{eq:metric-on-ldd}, 
    this implies that there are constants $w_\delta > 0$ and $Z_\delta > 0$ 
    such that $w_k(\mathcal{D'}) < w_\delta$
    and $\E[\sup_{n,t} Z_n(t)] < Z_\delta$.
    Then note that 
    for some $N \in \Nats$, 
    we have that 
    \begin{align*}
        \sum_{k >N} \E[\norm{w_k Z_k} \cdot \norm{e_k}_\infty + \norm{w'_k Z'_k} \cdot \norm{e'_k}_\infty] < \frac{\eps}{4}.
    \end{align*}
    With this $N$ fixed, 
    we can now proceed to the main part of the proof.
    Note that
    \begin{align*}
        \norm{H^{\rand}_{\mathcal{D}} - H^{\rand}_{\mathcal{D}'}}_\infty 
        &\le \sum_{k \ge 1} \sup_{\substack{t \in [0,1] \\ x \in M}}\left[\abs{{w_k Z_k(t) e_k(x) - w'_k Z'_k(t) e'_k(x)}}\right] \\
        &\leq 
        \sum_{k = 1}^N \norm{w_kZ_k - w'_kZ'_k}_\infty\cdot\norm*{\frac{e_k+e'_k}{2}}_\infty
        + \norm*{\frac{w_kZ_k + w'_kZ'_k}{2}}\cdot\norm{e_k - e'_k}_\infty \\
        &\qquad + \sum_{k >N} \sup_{\substack{t \in [0,1] \\ x \in M}}\abs{{w_k Z_k(t) e_k(x)}} +
        \sup_{\substack{t \in [0,1] \\ x \in M}}\abs{{ w'_k Z'_k(t) e'_k(x)}} \\
        &\leq 
        \sum_{k = 1}^N \norm{w_kZ_k - w'_kZ'_k}_\infty\cdot\norm*{\frac{e_k+e'_k}{2}}_\infty
        + \norm*{\frac{w_kZ_k + w'_kZ'_k}{2}}\cdot\norm{e_k - e'_k}_\infty \\
        &\qquad + \sum_{k >N} \norm{w_k Z_k} \cdot \norm{e_k}_\infty + \norm{w'_k Z'_k} \cdot \norm{e'_k}_\infty.
    \end{align*}
    It follows easily that 
    \begin{align*}
        \E[
        \norm{H^{\rand}_{\mathcal{D}} - H^{\rand}_{\mathcal{D}'}}_\infty]
        &\leq \frac{\eps}{4} +
        \sum_{k = 1}^N \E\left[\norm{w_kZ_k - w'_kZ'_k}_\infty\cdot\norm*{\frac{e_k+e'_k}{2}}_\infty\right]
        + \E\left[\norm*{\frac{w_kZ_k + w'_kZ'_k}{2}}\cdot\norm{e_k - e'_k}_\infty\right] \\
        &=\frac{\eps}{4} +
        \sum_{k = 1}^N \E\left[\norm{w_kZ_k - w'_kZ'_k}_\infty\right]\cdot\norm*{\frac{e_k+e'_k}{2}}_\infty
       + 
        \E\left[\norm*{\frac{w_kZ_k + w'_kZ'_k}{2}}\right]\cdot\norm{e_k - e'_k}_\infty\\
        &\leq 
        \frac{\eps}{4} +
        \sum_{k = 1}^N C'' \cdot N^{\frac{\dim M -1}{2}} \cdot \E\left[\norm{w_kZ_k - w'_kZ'_k}_\infty\right]
        + w_\delta \cdot Z_\delta \cdot\norm{e_k - e'_k}_\infty.
    \end{align*}
    Note that we can make $\delta > 0$ smaller without loss of generality.
    In particular, we might assume $\delta$ to be sufficiently small such 
    that 
    \begin{equation*}
         \E\left[\norm{w_kZ_k - w'_kZ'_k}_\infty\right] < \frac{\eps}{4C''N^{\frac{\dim M +1}{2}}} \qquad\text{and}\qquad \norm{e_k - e_k'}_\infty < \frac{\eps}{4Nw_\delta Z_\delta}
    \end{equation*}
    hold for any $1 \leq k \leq N$.
    This is true simply because for any single $k$ these are open conditions, and we therefore end up taking 
    a \emph{finite} intersection of open sets, which is again open.
    Combining all these estimates, we obtain that 
    \begin{align*}
        \E[
        \norm{H^{\rand}_{\mathcal{D}} - H^{\rand}_{\mathcal{D}'}}_\infty]
        \leq \frac{3}{4} \eps < \eps.
    \end{align*}
    This almost completes the proof.
    All that is left to do 
    is to notice that 
    \begin{equation*}
        W_p(\muGP^{\mathcal{D}},\muGP^{\mathcal{D}'})
        \le \left( \E\left[d_{\Hof}(\phi^{1}_{\randH_{\mathcal{D}}},\phi^{1}_{\randH_{\mathcal{D}'}})^p\right] \right)^{\frac{1}{p}}
        \le \left( 2^p \E\left[\norm{H^{\rand}_{\mathcal{D}} - H^{\rand}_{\mathcal{D}'}}_\infty \right]^p \right)^{\frac{1}{p}}
        < 2\eps.
    \end{equation*}
    At the end we use the fact that the map sending Hamiltonian functions to Hamiltonian diffeomorphisms is $2$-Lipschitz with respect to the Hofer metric and the $L^\infty$-norm, see e.g.~\cite[Lemma 3.15]{dawid-2025} for a proof.
    Thus, we have obtained that whenever $d(\mathcal{D},\mathcal{D}') < \delta$, we have 
    $W_p(\muGP^{\mathcal{D}},\muGP^{\mathcal{D}'}) < 2\eps$.
    Since $\eps > 0$ and $\mathcal{D} \in LDD$ were arbitrary, 
    this implies that the map $\mathcal{D} \mapsto \muGP^{\mathcal{D}}$ is continuous, and thereby completes the proof.
\end{proof}
\begin{remark}
    Note that the projection to the almost-complex structure is a continuous map $\LDD(M,\omega) \to \mathcal{J}(M,\omega)$.
    The latter space is contractible, but the fibers of this map are not easy to understand. One natural question 
    is whether one can always lift a path.
    Assume $\{J_t\}_{t \in [0,1]}$ is a smooth path of almost-complex structures.
    Then, we can ask whether there exists a smooth path $\{\mathcal{D}_t\}_{t \in [0,1]}$ of law-defining data such that the projection to the almost-complex structure is $J_t$ for all $t \in [0,1]$.
    As we deform the almost-complex structure $J_{t}$ smoothly, the induced Riemannian metric $\omega(\cdot,J_{t}\cdot)$ varies smoothly as well.
    It is well-known that the eigenvalues vary continuously, thus we can choose eigenvalues $\lambda_1(t) \le \lambda_2(t) \le \ldots$ for all $t \in [0,1]$ such that $\lambda_n(t)$ is continuous in $t$ for all $n \ge 1$.
    In general, the eigenfunctions $e_n(t)$ cannot be chosen to vary continuously in $t$ if the correct ordering is to be preserved.
    This is due to the fact that the eigenspaces of $\Laplace_{J_{t}}$ can have dimension greater than one, and two eigenvalues might collide at some $t \in [0,1]$.
    If this does not occur, we can clearly lift the path by choosing the eigenfunctions $e_n(t)$ to vary continuously in $t$ for all $n \ge 1$
    and keeping the coefficient processes $\{Z_n\}_{n \ge 1}$ fixed.
    If we cannot track the eigenfunctions continuously, lifting is not possible in general.
    We can see this by choosing a path of almost-complex structures such that the eigenvalues switch places along the isotopy.
    Thus, by the ordering imposed on the eigenvalues, the eigenfunctions cannot be chosen to vary continuously in $t$.
    If we take a law-defining datum where a single coefficient process is non-zero, then the law of the random Hamiltonian function will 
    jump when that switch takes place.
\end{remark}
\section{Measurability of further topologies on $\Ham(M,\omega)$}\label{sec:topologies}
In the first part of this paper, the focus has been on proving 
that the measures constructed from law-defining data are Borel measures with respect to the Hofer topology on $\Ham(M,\omega)$.
However, there are other topologies on $\Ham(M,\omega)$ that are of interest in symplectic geometry, such as the $C^k$-topologies for $k \in \Nats_0$ and the $C^\infty$-topology.
In this section, we will show that the measures constructed from law-defining data are also Borel measures with respect to these topologies.
Recall that the measures in general are created by pushing forward a probability measure from an abstract probability space to $\Ham(M,\omega)$.
Thus, the extension to further topologies on $\Ham(M,\omega)$ should \emph{not} be understood as creating new measures, but rather as
extending the class of measurable events.
Let $\mathcal{D}$ be a \ldd , and
$(\Omega, \mathcal{F}, \Prob)$ the abstract probability space on which the associated random Hamiltonian function $\randH$ is defined.
Then $\phi^1_{\randH}\colon \Omega \to \Ham(M,\omega)$ is the random variable that induces the measure $\muGP^{\mathcal{D}}$ on $\Ham(M,\omega)$.
In~\cite{dawid-2025},
it was shown that $\phi^1_{\randH}$ is measurable with respect to the Hofer topology on $\Ham(M,\omega)$.
We now extend this result to the $C^k$-topologies for $k \in \Nats_0$ and the $C^\infty$-topology.
Note that the value of any observable is independent of the topology on $\Ham(M,\omega)$, and thus the expectation of any observable is also independent of the topology.
The following diagram is (trivially) commutative:
\begin{center}
    \begin{tikzcd}
                                                                                                                                                                                                                                        &  & {(\Ham(M, \omega),d_{\Hof})}                                                     \\
                                                                                                                                                                                                                                        &  &                                                                          \\
    {(\Omega,\mathcal{F})} \arrow[rr, "\phi^1_{\randH}"] \arrow[rrd, "\phi^1_{\randH}"] \arrow[rrdd, "\phi^1_{\randH}"] \arrow[rrdddd, "\phi^1_{\randH}"] \arrow[rrddddd, "\phi^1_{\randH}", bend right] \arrow[rruu, "\phi^1_{\randH}"] &  & {(\Ham(M, \omega),d_{C^\infty})} \arrow[d, hook]                      \\
                                                                                                                                                                                                                                        &  & {(\Ham(M, \omega),d_{C^k})} \arrow[d, hook]                                      \\
                                                                                                                                                                                                                                        &  & {(\Ham(M, \omega),d_{C^{k-1}})} \arrow[d, hook]                                  \\
                                                                                                                                                                                                                                        &  & \dots \arrow[d, hook]                                                    \\
                                                                                                                                                                                                                                        &  & {(\Ham(M, \omega), d_{C^1})} \arrow[d, hook] \arrow[uuuuuu, hook, bend right=60] \\
                                                                                                                                                                                                                                        &  & {(\Ham(M, \omega),d_{C^0})}                                                     
    \end{tikzcd}
\end{center}
For any $k \in \Nats_0 \cup \{\infty\}$, $1 \leq m \leq k$, the inclusion map $(\Ham(M, \omega),d_{C^{k}}) \hookrightarrow (\Ham(M, \omega),d_{C^m})$ is continuous.
Thus, it suffices to show that $\phi^1_{\randH}$ is measurable with respect to the $C^\infty$-topology on $\Ham(M,\omega)$.
In practice, this is established by proving the measurability of $\phi^1_{\randH}$ with respect to all $C^k$-topologies for $k \in \Nats$.
The two endpoints of the diagram, $(\Ham(M, \omega),d_{\Hof})$ and $(\Ham(M, \omega),d_{C^0})$, are truly different,
as it is well-known that the Hofer topology and the $C^0$-topology on $\Ham(M,\omega)$ are not comparable.
\subsection{Sobolev bounds on the random Hamiltonian function}
Let $\mathcal{D}$ be a \ldd , and let $\randH$ be the associated random Hamiltonian function.
Recall from~\eqref{eq:GP_H} that $\randH$ is given by
\begin{equation*}
    \randH(t,x) = \sum _{n \ge 1} w_n \cdot Z_n(t) \cdot e_n(x).
\end{equation*} 
By~\cite[Lemma 3.8]{dawid-2025}, we have that $\randH$ almost-surely 
lies in the Sobolev space $W^{2s,2}([0,1] \times M)$ for any $s \in \Nats_0$.
Thus, we can see $\randH$ as a random variable taking values in the Sobolev space $W^{2s,2}([0,1] \times M)$ for any $s \in \Nats_0$.
Now fix $s \in \Nats_0$ large enough such that $W^{2s,2}([0,1] \times M)$ embeds continuously into $C^{k+1}([0,1] \times M)$ for some $k \in \Nats_0$.
Then, the random variable $\randH$ is measurable with respect to the $C^{k+1}$-topology on $C^\infty_0([0,1] \times M)$.
\subsection{Measurability of the flow with respect to the $C^k$- and $C^\infty$-topologies}
Consider the map 
\begin{align*}
     C^\infty_0([0,1] \times M) &\to \Ham(M,\omega) \\
    H &\mapsto \phi^1_H.
\end{align*}
Due to the smooth dependence of an ODE on its initial conditions and parameters, this map is continuous with respect to the $C^k$-topology on $\Ham(M,\omega)$ and the $C^{k+1}$-topology on $C^\infty_0([0,1] \times M)$ for any $k \in \Nats_0$.
Thus, if the random Hamiltonian function $\randH$ is measurable with respect to the $C^{k+1}$-topology on $C^\infty_0([0,1] \times M)$, then the random variable $\phi^1_{\randH}$ is measurable with respect to the $C^k$-topology on $\Ham(M,\omega)$.
However, we have already seen that $\randH$ is measurable with respect to the $C^{k+1}$-topology on $C^\infty_0([0,1] \times M)$ for any $k \in \Nats_0$
in the previous subsection.
Thus, we can conclude that $\phi^1_{\randH}$ is measurable with respect to the $C^k$-topology on $\Ham(M,\omega)$ for any $k \in \Nats_0$.
Finally, since the $C^\infty$-topology on $\Ham(M,\omega)$ is the projective limit of the $C^k$-topologies for $k \in \Nats_0$, we can conclude that $\phi^1_{\randH}$ is also measurable with respect to the $C^\infty$-topology on $\Ham(M,\omega)$.
This proves Theorem~\ref{thm:measurability_Ck}.
\begin{remark}
    This extension considerably enlarges the class of measurable events on $\Ham(M,\omega)$, and thus the class of observables that can be considered.
    However, there is still a property that makes the Hofer topology on $\Ham(M,\omega)$ special: the map from $C^\infty_0([0,1] \times M)$ to $\Ham(M,\omega)$ that sends a Hamiltonian function to its time-one flow is Lipschitz with respect to the Hofer metric on $\Ham(M,\omega)$ and the $L^\infty$-norm on $C^\infty_0([0,1] \times M)$. Thus, we can obtain quantitative estimates on the measure of Hofer balls (see~\cite{dawid-2025} for details), which is not possible for the $C^k$-topologies on $\Ham(M,\omega)$.
    This is due to the fact that Lipschitz functions are compatible with concentration of measure phenomena, which are the key to these results.
    While the maps above are still continuous, they do not result in similar sub-Gaussian type quantitative estimates on the measure of $C^k$-balls in $\Ham(M,\omega)$.
\end{remark}
\section{Statistics of the commutator length and stable commutator length}\label{sec:scl-stats}
The fact that we have extended the measure $\muGP^{\mathcal{D}}$ to the Borel $\sigma$-algebra of the $C^\infty$-topology allows us to study the statistics of group-theoretic invariants of $\Ham(M,\omega)$, such as the commutator length and the stable commutator length.
We will show that any homogeneous quasimorphism on $\Ham(M,\omega)$ is square-integrable with respect to $\muGP^{\mathcal{D}}$.
In recent deep work, Edtmair has shown that the commutator length on $\Ham(M,\omega)$ is locally bounded with respect to the $C^\infty$-topology.
We restate Edtmair's result from~\cite{edtmair-2025} for the convenience of the reader:
\begin{theorem}[Edtmair]\label{thm:edtmair-2025}
Let $(M^{2n},\omega)$ be a closed symplectic manifold.
Then there exists an integer $m>0$ depending only on $n$ such that there exist a $C^\infty$-open neighborhood $\mathcal{N} \subset \Ham(M,\omega)$ of the identity and a smooth map
\begin{equation*}
\Psi\colon \mathcal{N} \to \Ham(M,\omega)^{2m}
\end{equation*}
which is a local right inverse of 
the map 
\begin{align*}
\Phi\colon\Ham(M,\omega)^{2m} &\to\Ham(M,\omega) \\ (u_1,v_1,\dots,u_m,v_m) &\mapsto \prod\limits_{j=1}^m [u_j,v_j],
\end{align*}
 i.e., $\Phi\circ \Psi = \id_{\mathcal{N}}$. We can choose $\mathcal{N}$ and $\Psi$ such that $\Psi(\id)$ is arbitrarily close to the tuple $(\id,\dots,\id)$.
\end{theorem}
In the above, $\Ham(M,\omega)$ is endowed with the $C^\infty$-topology.
Recall that this is the Fréchet topology generated by the family of $C^k$-norms for $k \in \Nats$.
Thus, for a given $(M^{2n},\omega)$ and $\mathcal{N} \subset \Ham(M,\omega)$, as in Theorem~\ref{thm:edtmair-2025} above, 
there exists an integer $K > 0$ and some $\eps_{K} > 0$ such that the $C^K$-ball of radius $\eps_K$ around the identity is contained in $\mathcal{N}$.
In particular, if $H$ is a Hamiltonian function with a sufficiently small $C^{K+1}$-norm, then the time-1 map of the flow generated by $H$ is contained in $\mathcal{N}$ and hence can be written as a product of $2m$ commutators.
This result will be at the heart of this section.

\subsection{Sub-Gaussian tail estimates for the commutator length}
Recall that $[\Ham(M,\omega),\Ham(M,\omega)] = \Ham(M,\omega)$, i.e., any element of $\Ham(M,\omega)$ can be written as a product of commutators.
The commutator length $\cl\colon \Ham(M,\omega) \to \Z$ is then defined 
as the minimal number of commutators needed to express a given element of $\Ham(M,\omega)$ as a product of commutators.
We will now show that the likelihood of a random Hamiltonian diffeomorphism having a large commutator length is exponentially small.
We first establish the following lemma, which we will use several times in this section.
\begin{lemma}\label{lem:cl-gaussian-tail}
Let $(M^{2n},\omega)$ be a closed symplectic manifold 
and $\mathcal{D}$ a \ldd .
Then the commutator length $\cl\colon \Ham(M,\omega) \to \Z$ is bounded from above by a measurable function $f$ 
which is such that $\muGP^{\mathcal{D}}(f^{-1}([t,\infty)))$ has sub-Gaussian tails, i.e., there exist constants $C,c > 0$ such that
\begin{align*}
    \muGP^{\mathcal{D}}(f^{-1}([t,\infty))) \leq C e^{-ct^2},
\end{align*}
for all $t \geq 0$.
In particular, $f$ is square-integrable with respect to $\muGP^{\mathcal{D}}$.
\end{lemma}
\begin{proof}
    To show this lemma, it is more convenient to work at the level of Hamiltonian functions instead of at the level of Hamiltonian diffeomorphisms.
    Thus, we wish to look at the following composition:
    \begin{center}
        \begin{tikzcd}
            {C^\infty([0,1]\times M)} \arrow[r] & \Ham(M, \omega)  \arrow[r]           & \Z           \\
            H \arrow[r, maps to]                & \phi_H^1 \arrow[r, maps to] & \cl(\phi_H^1).
        \end{tikzcd}
    \end{center}
    Now let $m \in \Nats$ be the integer from Theorem~\ref{thm:edtmair-2025} and let $K > 0$ and $\eps_K > 0$ be such that the $C^K$-ball of radius $\eps_K$ around the identity is contained in $\mathcal{N}$, where $\mathcal{N}$ is the neighborhood of the identity from Theorem~\ref{thm:edtmair-2025}.
    Note that the map $H \mapsto \phi_H^1$ is continuous with respect to the $C^{k+1}$-topology on $C^\infty([0,1]\times M)$ and the $C^k$-topology on $\Ham(M,\omega)$ for any $k \in \Nats$.
    Thus, there exists some $\delta_K > 0$ such that if $H$ is a Hamiltonian function with $\norm{H}_{C^{K+1}} \leq \delta_K$, then the time-1 map of the flow generated by $H$ is contained in $\mathcal{N}$ and hence can be written as a product of $2m$ commutators.
    Given any $H \in C^\infty([0,1]\times M)$, we can split $H$ into small pieces, each of which has $C^{K+1}$-norm less than $\delta_K$.
    The construction works as follows:
    let $\beta\colon [0,1] \to[0,1]$ be a smooth function such that $\beta(t) = 0$ for $t$ close to $0$, $\beta(t) = 1$ for $t$ close to $1$ and $\beta$ is monotone.
    Set $N \coloneqq \lceil \sup_{t \in [0,1]} \abs{\beta'(t)}\rceil \cdot \lceil \norm{H}_{C^{K+1}}/\delta_K \rceil$.
    Then let 
    \begin{align*}
        H_j(t,x) \coloneqq \frac{\beta'(t)}{N} \cdot H\left(\frac{j+\beta(t)}{N},x\right),
    \end{align*}
    for $j = 0,\dots,N - 1$.
    If we consider any trajectory $\gamma_j$ of the Hamiltonian vector field of $H_j$ starting at $x \in M$ at time $0$, then $\gamma_j$ is a reparametrization of the trajectory of the Hamiltonian vector field of $H$ starting at $x$ at time $\frac{j}{N}$ up to time $\frac{j+1}{N}$.
    Thus, the time-1 map of $H_j$ is given by 
    \begin{equation*}\phi_{H_j}^1 = \phi_H^{\frac{j+1}{N}} \circ (\phi_H^{\frac{j}{N}})^{-1}.\end{equation*}
    In particular, we have that
    \begin{align*}
        \phi_H^1 = \phi^1_{H_{N-1}} \circ \ldots \circ \phi^1_{H_0}.
    \end{align*}
    One easily verifies that  
    \begin{align*}
        \norm{H_j}_{C^{K+1}} \leq \frac{\sup_{t \in [0,1]} \abs{\beta'(t)}}{N} \cdot \norm{H}_{C^{K+1}} \leq \frac{1}{\lceil \norm{H}_{C^{K+1}}/\delta_K \rceil} \cdot \norm{H}_{C^{K+1}} \leq \delta_K,
    \end{align*}
    for all $j$, and hence each $\phi_{H_j}^1$ can be written as a product of $2m$ commutators.
    Thus, we obtain that 
    \begin{equation}\label{eq:cl-bound}
        \cl(\phi_H^1) \leq 2m \cdot \lceil \sup_{t \in [0,1]} \abs{\beta'(t)} \rceil \cdot \lceil \norm{H}_{C^{K+1}}/\delta_K \rceil.
    \end{equation}
    Note that both $m$ and $\delta_K$ depend only on $M$ and $\omega$, and not on $H$. 
    The constant $\lceil \sup_{t \in [0,1]} \abs{\beta'(t)} \rceil$ is universal and can assumed to be $2$ without loss of generality by choosing $\beta$ appropriately.
    Let $\randH\colon [0,1] \times M \to \Reals$ be the random Hamiltonian function obtained from the construction of the random Hamiltonian flow with respect to $\mathcal{D}$.
    Then, $\norm{\randH}_{C^{K+1}}$ is a random variable with sub-Gaussian tails and thus is in particular square-integrable with respect to $\muGP^{\mathcal{D}}$.
    To see why, let $2k > 0$ be sufficiently large such that there is a
    Sobolev embedding $W^{2k,2}([0,1]\times M) \into C^{K+1}([0,1]\times M)$.
    Then, by the Sobolev embedding theorem, there exists a constant $C > 0$ such that
    \begin{align*}
        \norm{\randH}_{C^{K+1}} \leq C \norm{\randH}_{W^{2k,2}}.
    \end{align*}
    Then, by the definition of the $W^{2k,2}$-norm and~\eqref{eq:GP_H}, 
    we can see that $\norm{\randH}_{W^{2k,2}}$ is a sub-Gaussian random variable.
    Hence, the right-hand side of the inequality in~\eqref{eq:cl-bound} is a function with the necessary sub-Gaussian tail estimate,
    which bounds $\cl(\phi_{\randH}^1)$ from above.
\end{proof}
Unfortunately, the commutator length itself is not necessarily measurable with respect to the Borel $\sigma$-algebra of the $C^\infty$-topology, 
and thus does not constitute a random variable with respect to $\muGP^{\mathcal{D}}$.
However, the following slightly weaker statement still establishes a kind of sub-Gaussian tail estimate for the commutator length.
The following corollary follows directly from Lemma~\ref{lem:cl-gaussian-tail}.
\begin{corollary}\label{cor:cl-gaussian-tail}
Let $(M^{2n},\omega)$ be a closed symplectic manifold and $\mathcal{D}$ a \ldd .
Then there exists a sequence of Borel measurable sets $\mathcal{A}_k \subset \Ham(M,\omega)$ with respect to the $C^\infty$-topology such that $\mathcal{A}_k \subset \{ \phi \in \Ham(M,\omega) \mid \cl(\phi) \geq k \}$  and there exist constants $C,c > 0$ such that $\muGP^{\mathcal{D}}(\mathcal{A}_k) \leq C e^{-ck^2}$ for all $k \geq 0$.
Thus, the probability that the commutator length is large decays at least as fast as that of a Gaussian being large.
\end{corollary}
\subsection{The stable commutator length and homogeneous quasimorphisms}
A more well-behaved group-theoretic invariant of $\Ham(M,\omega)$ is the stable commutator length $\scl$.
Recall that $\scl(g) = \lim_{n \to \infty} \frac{\cl(g^n)}{n}$ for $g \in [G,G]$ on a group $G$.
This is a real-valued function and --- unlike the commutator length -- can easily be shown to be measurable with respect to the Borel $\sigma$-algebra of the $C^\infty$-topology.
\begin{lemma}
Let $(M^{2n},\omega)$ be a closed symplectic manifold and $\mathcal{D}$ a \ldd .
Then the stable commutator length $\scl\colon \Ham(M,\omega) \to [0,\infty)$ is measurable with respect to the Borel $\sigma$-algebra of the $C^\infty$-topology.
\end{lemma}
\begin{proof}
    To show this lemma, we will use \emph{Bavard duality},
    which relates the stable commutator length $\scl$ to homogeneous quasimorphisms as follows: for any $g \in [G,G]$, we have that
    \begin{align*}
        \scl(g) = \sup_{q} \frac{q(g)}{2D(q)},
    \end{align*}
    where the supremum is taken over all homogeneous quasimorphisms $q$ on $G$ and $D(q)$ is the defect of $q$.
    We use the convention that $\sup \emptyset = 0$, and thus $\scl(g) = 0$ if there is no homogeneous quasimorphism $q$ with $q(g) \neq 0$.
    Edtmair proved that any homogeneous quasimorphism on $\Ham(M,\omega)$ is continuous with respect to the $C^\infty$-topology, see~\cite[Corollary 1.4]{edtmair-2025}.
    Now, by Bavard duality, we have that $\scl$ is the supremum of a family of continuous functions with respect to the $C^\infty$-topology, and hence is lower-semicontinuous with respect to the $C^\infty$-topology. In particular, for any $s \in \Reals$, the set $\{g \in \Ham(M,\omega) \mid \scl(g) \leq s\}$ is closed with respect to the $C^\infty$-topology, and hence is measurable with respect to the Borel $\sigma$-algebra of the $C^\infty$-topology.
    Since the sets $(-\infty, s]$ generate the Borel $\sigma$-algebra of $\mathbb{R}$, the claim follows.
\end{proof}
Since $\scl$ is bounded from above by the commutator length, Lemma~\ref{lem:cl-gaussian-tail} implies that $\scl$
itself is a genuine random variable with sub-Gaussian tails with respect to $\muGP^{\mathcal{D}}$:
\begin{corollary}\label{cor:scl-gaussian-tail}
Let $(M^{2n},\omega)$ be a closed symplectic manifold and $\mathcal{D}$ a \ldd .
Then the stable commutator length $\scl\colon \Ham(M,\omega) \to [0,\infty)$ is a sub-Gaussian random variable with respect to the probability measure $\muGP^{\mathcal{D}}$.
\end{corollary}
We obtain one further result that 
follows in the same way as Corollary~\ref{cor:scl-gaussian-tail}.
Recall that Edtmair has shown that any homogeneous quasimorphism on $\Ham(M,\omega)$ is continuous with respect to the $C^\infty$-topology, see~\cite[Corollary 1.4]{edtmair-2025}.
Thus, by Theorem~\ref{thm:measurability_Ck}, any homogeneous quasimorphism on $\Ham(M,\omega)$ is a random variable with respect to $\muGP^{\mathcal{D}}$.
Now, note the following:
given any element $g \in \Ham(M,\omega)$ and any homogeneous quasimorphism $q$ on $\Ham(M,\omega)$, we have that
\begin{align*}
    |q(g)| \leq 2D(q) \cdot \cl(g),
\end{align*}
where $D(q)$ is the defect of $q$. This follows simply from writing $g$ as a product of commutators and using the definition of a quasimorphism.
Thus, we obtain the following corollary:
\begin{corollary}\label{cor:quasimorphism-gaussian-tail}
Let $(M^{2n},\omega)$ be a closed symplectic manifold and $\mathcal{D}$ a \ldd .
Then any homogeneous quasimorphism $q$ on $\Ham(M,\omega)$ is a sub-Gaussian random variable with respect to the probability measure $\muGP^{\mathcal{D}}$.
In particular, any homogeneous quasimorphism on $\Ham(M,\omega)$ is square-integrable with respect to $\muGP^{\mathcal{D}}$.
\end{corollary}
\section{A central limit theorem for Lipschitz quasimorphisms}\label{sec:clt}
In this section, we will show Theorem~\ref{thm:q-clt}, i.e., a central limit theorem for Lipschitz quasimorphisms on a group with a right-invariant metric.
The proof is based on~\cite{bjorklund-hartnick-2011}, where 
the authors show a central limit theorem for quasimorphisms on locally compact groups.\footnote{
For the sake of completeness, we note that~\cite{bjorklund-hartnick-2011} also contains a proof of the central limit theorem for general measured groups under the assumption of the continuum hypothesis. However, we neither want to nor need to assume the continuum hypothesis here.}
 The overall strategy is to show that one can find a suitable ergodic martingale that is at bounded 
distance from a given quasimorphism, and then apply a central limit theorem for such martingales due to Billingsley in~\cite{billingsley-1961}.
One can also apply the main result of~\cite{stout-1970} to obtain a law of the iterated logarithm based on the same martingale.
The key step hinges on the following definition:
\begin{definition}
Let $(G,\mu)$ be a measured group. A real-valued function $q$ on $G$ is \emph{quasi-right-harmonic with
respect to $\mu$} if it is Borel measurable and integrable with respect to the measure $\mu$ and there is a constant $\ell$ such that
\begin{equation*}
\int_G q(gh) \, d\mu(h) = q(g) + \ell,
\end{equation*}
for all $g$. If 
\begin{equation*}
\int_G q(hg) \, d\mu(h) = q(g) + \ell,
\end{equation*}
then $q$ is \emph{quasi-left-harmonic with respect to $\mu$}. If $q$ is both quasi-right-harmonic and quasi-left-harmonic, then we say that $q$ is \emph{quasi-bi-harmonic with respect to $\mu$}. If $\ell = 0$ in any of the above definitions, then we say that $q$ is \emph{right-harmonic}, \emph{left-harmonic} or \emph{bi-harmonic}, respectively.
\end{definition}
Obtaining a suitable martingale then boils down to showing that there is a quasi-bi-harmonic function at bounded distance from the given quasimorphism.
By replacing~\cite[Proposition 2.2]{bjorklund-hartnick-2011} by the following lemma, we can adapt the proof of Björklund and Hartnick to our setting.
\begin{lemma}\label{lem:bi-harmonic-q}
	Let $(G,d_G)$ be a metric group endowed with a right-invariant metric and let $q\colon G \to \Reals$ be a homogeneous $d_G$-Lipschitz quasimorphism.
	Further, let $\mu$ be a Borel probability measure on $G$.
	Then there exists a bi-harmonic map $\tilde{q}\colon G \to \Reals$ that is at bounded distance from $q$.
\end{lemma}
\begin{proof}
	Let $\ul{\lambda}\colon \ell^\infty \to \Reals$ be a Banach limit, i.e., a positive linear functional on $\ell^\infty$ with norm one that is shift-invariant and extends the usual limit on convergent sequences.
	Recall that for a quasimorphism $q$, the defect $\partial q$ is defined by $\partial q(g,h) = q(gh) - q(g) - q(h)$ for all $g,h \in G$.
	By definition of a quasimorphism, $\partial q$ is uniformly bounded.
	Using this, we define 
	\begin{equation*}\tilde{q}_n(g_1,g_2) = \int_G \partial q(g_1,g_2 h) \, d\mu^{*n}(h).\end{equation*}
	Note that since $\partial q$ is uniformly bounded, the functions $\tilde{q}_n$ are uniformly bounded on $G \times G$. 
	Then, the function
	\begin{equation*}\tilde{q}(g) = q(g) + \ul{\lambda}(\{\tilde{q}_n(g,e)\}_{n \in \Nats})\end{equation*}
	is well-defined and at bounded distance from $q$. 
	We will now show that it is measurable and bi-harmonic.

	We first claim that $\tilde{q}$ is measurable. To see this, we first note that 
	$q$ is measurable, and that $\ul{\lambda}$ is a bounded linear operator with norm one.
	Further, we show that all $\tilde{q}_n(g,e)$ are Lipschitz.
	Let $L_q$ be the Lipschitz constant of $q$, 
	and let $g,g' \in G$. Then, we have 
	\begin{align*}
		\abs{\tilde{q}_n(g,e) - \tilde{q}_n(g',e)} &= \abs*{\int_G \partial q(g,h) - \partial q(g',h) \, d\mu^{*n}(h)} 
		\leq \int_G \abs{\partial q(g,h) - \partial q(g',h)} \,d\mu^{*n}(h) \\
		&= \int_G \abs{q(gh)-q(g'h)+q(g)-q(g')} \,d\mu^{*n}(h) \\
		&\leq \int_G L_q ( d_G(gh,g'h)+d_G(g,g') ) \,d\mu^{*n}(h) \\
		&= 2 L_q d_G(g,g'),
	\end{align*}
	where we use the fact that $\mu^{*n}$ is a probability measure and that the metric on $G$ is right-invariant, 
	i.e., $d_G(gh,g'h) = d_G(g,g')$.
	Thus, the map from $G$ to $\ell^\infty$ given by $g \mapsto \{\tilde{q}_n(g,e)\}_{n \in \Nats}$ is Lipschitz and hence continuous.
	As $\ul{\lambda}$ is a bounded linear operator, the composition $g \mapsto \ul{\lambda}(\{\tilde{q}_n(g,e)\}_{n \in \Nats})$ is continuous and hence measurable.
	Thus, $\tilde{q}$ is measurable as the sum of two measurable functions.
	Notice further that, by construction, $\tilde{q}$ is at bounded distance from $q$ since all $\tilde{q}_n$ are uniformly bounded by the defect of $q$.

	Next, we need to show that $\tilde{q}$ is indeed bi-harmonic. 
	Note that 
	\begin{align*}
		\int_G \tilde{q}_n(g,h) \, d\mu(h) 
		&=
		\int_G \int_G \big( q(ghk) - q(hk) - q(g) \big) \, d\mu^{*n}(k) \, d\mu(h) \\
		&=
		\int_G \big( q(gh) - q(h) - q(g) \big) \, d\mu^{*(n+1)}(h) \\
		&=
		\tilde{q}_{n+1}(g,e),
	\end{align*}
	for all $n$. Note that by the shift-invariance of $\ul{\lambda}$, we have
		$\ul{\lambda}(\{\tilde{q}_n(g,e)\}_{n \in \Nats})
		=
		\ul{\lambda}(\{\tilde{q}_{n+1}(g,e)\}_{n \in \Nats})$.
		Furthermore, note that by linearity of $\ul{\lambda}$, we have that $\ul{\lambda}$ and $\int_G \cdot \, d\mu(h)$ commute, so that we can conclude that
		\begin{align*}
		\int_G \ul{\lambda}(\{\tilde{q}_n(g,h)\}_{n \in \Nats}) \, d\mu(h)
		&=
		\ul{\lambda}(\{\int_G \tilde{q}_n(g,h) \, d\mu(h)\}_{n \in \Nats}) \\
		&=
		\ul{\lambda}(\{\tilde{q}_{n+1}(g,e)\}_{n \in \Nats}) \\
		&=
		\ul{\lambda}(\{\tilde{q}_n(g,e)\}_{n \in \Nats}).
	\end{align*}
	Note that 
	\begin{align*}
	\tilde{q}_n(gk,e) &= \int_G \partial q(gk,h) \, d\mu^{*n}(h)
	= \int_G q(gkh) - q(gk) - q(h) \, d\mu^{*n}(h) \\
	&= \int_G \left[q(ghk) - q(g) - q(kh)\right] + \left[q(kh) -q(k) - q(h) \right] + q(g) + q(k) - q(gk)\, d\mu^{*n}(h) \\
	&= \tilde{q}_n(g,k) + \tilde{q}_n(k,e) + q(g) + q(k) - q(gk)
	\end{align*}
	holds for all group elements $g, k \in G$ and any $n \in \Nats$.
	Recall that \emph{homogeneous} quasimorphisms are conjugation-invariant.
	This can be easily seen by the standard computation that
	\begin{equation*}
		\abs{q(hgh^{-1}) - q(g)} = \frac{\abs{q(hg^kh^{-1}) - q(g^k)}}{k} = \frac{\abs{q(hg^kh^{-1}) - q(h) - q(g^k) - q(h^{-1})}}{k} \leq 2 \frac{D(q)}{k},
	\end{equation*}
	where $D(q) = \norm{\partial q}_\infty$ is the defect of $q$ and $k$ is an arbitrary positive integer. By letting $k$ go to infinity, we obtain that $q(hgh^{-1}) = q(g)$ for all $g,h \in G$.
	Thus, we have 
	\begin{align*}
		\tilde{q}_n(kg,e) &= \int_G q(kgh) - q(kg) - q(h) \, d\mu^{*n}(h)
		= \int_G q(ghk) - q(kg) - q(h) \, d\mu^{*n}(h)
	\end{align*}
	since $kgh$ and $ghk$ are conjugate. In the same manner as before, we obtain:
	\begin{equation*}
		\tilde{q}_n(kg,e) = \tilde{q}_n(g,k) + \tilde{q}_n(k,e) + q(g) + q(k) - q(kg).
	\end{equation*}
	These two formulas allow us to directly check that $\tilde{q}$ is quasi-right-harmonic and quasi-left-harmonic, respectively.
	We start by showing the quasi-right-harmonicity of $\tilde{q}$.
	Our previous computations allow us to conclude for all $g \in G$ that
	\begin{align*}
		\int_G \tilde{q}(gk) \, d\mu(k) 
		&=
		\int_G q(gk) + \ul{\lambda}(\{\tilde{q}_n(gk,e)\}_{n \in \Nats}) \, d\mu(k) \\
		&=
		\int_G q(gk) \, d\mu(k) + \ul{\lambda}\left( \int_G \{\tilde{q}_n(gk,e)\}_{n \in \Nats} \, d\mu(k)\right) \\
		&=
		\int_G q(gk) \, d\mu(k) + \ul{\lambda}\left(\int_G\{\tilde{q}_n(g,k) + \tilde{q}_n(k,e) + q(g) + q(k) - q(gk)\}_{n \in \Nats}
 \, d\mu(k)\right) \\
		&=
		q(g) + \int_G q(k) \, d\mu(k) + \ul{\lambda}\left(\int_G\{\tilde{q}_n(g,k) + \tilde{q}_n(k,e)\}_{n \in \Nats}
 \, d\mu(k) \right) \\
		&=
		q(g) +\ul{\lambda}\left(\{\tilde{q}_n(g,e)\}_{n \in \Nats} \right) + \int_G q(k) + \ul{\lambda}\left( \{\tilde{q}_n(k,e)\}_{n \in \Nats} \right) \, d\mu(k) \\
		&=
		\tilde{q}(g) + \int_G \tilde{q}(k) \, d\mu(k).
	\end{align*}
	Thus, we have shown that $\tilde{q}$ is quasi-right-harmonic. We now need to show the quasi-left-harmonicity of $\tilde{q}$. 
	By the same method, we obtain 
	\begin{align*}
		\int_G \tilde{q}(kg) \, d\mu(k)  
		&= 
		\int_G q(kg) \, d\mu(k) + \ul{\lambda}\left( \int_G \{\tilde{q}_n(kg,e)\}_{n \in \Nats}\right) \, d\mu(k) \\
		&=
		\int_G q(kg) \, d\mu(k) + \ul{\lambda}\left( \int_G \{ \tilde{q}_n(g,k) + \tilde{q}_n(k,e) + q(g) + q(k) - q(kg) \}_{n \in \Nats}\right) \, d\mu(k) \\
		&= \tilde{q}(g) + \int_G \tilde{q}(k) \, d\mu(k).
	\end{align*} 
	This completes the proof that $\tilde{q}$ is bi-harmonic. 
\end{proof}
We obtain Theorem~\ref{thm:q-clt} using the proof of~\cite[Theorem 1.11]{bjorklund-hartnick-2011}
after replacing the use of~\cite[Proposition 2.2]{bjorklund-hartnick-2011} by the above lemma. 
Since the Hofer norm is in particular right-invariant, this theorem directly applies to Hofer-Lipschitz homogeneous quasimorphisms on $\Ham(M,\omega)$ or $\widetilde{\Ham}(M,\omega)$.
We can immediately apply this central limit theorem to obtain the following growth estimate for the expected value of a Hofer-Lipschitz homogeneous quasimorphism under the random walk induced by a centered autonomous law-defining datum.
\begin{theorem}\label{thm:abs-q-growth}
    Let $(M,\omega)$ be a closed symplectic manifold and $\mathcal{D}$ a centered 
	 autonomous law-defining datum on $M$ such that there is some Hofer-Lipschitz homogeneous quasimorphism $q\colon \Ham(M,\omega) \to \Reals$ that is not $\muGP^{\mathcal{D}}$-tame.
    Then there exist constants $0 \leq c \leq C < \infty$ such that $c\sqrt{n} \leq \E[\abs{q(\Phi_n)}] \leq C\sqrt{n}$ for all $n$, where $\Phi_n$ is the random walk induced by $\mathcal{D}$. 
	Further, we have that there exists $N \in \Nats$ such that for all $n \geq N$, we have $\E[\abs{q(\Phi_n)}] \geq c'\sqrt{n}$ for some constant $c' > 0$.
\end{theorem}
\begin{proof}
    By Theorem~\ref{thm:q-clt}, 
    the distribution of the random variable $\frac{q(\Phi_n)}{\sqrt{n}}$ converges to a (non-degenerate) Gaussian distribution as $n \to \infty$. 
    Since $q$ is homogeneous and $\muGP^{\mathcal{D}}$ is symmetric, i.e., invariant under inversion, 
    we have that $\E[q(\Phi_n)] = 0$ for all $n$.
	Here we use the fact that $\mathcal{D}$ is centered, which implies that $\muGP^{\mathcal{D}}$ is symmetric.
	Thus, the limiting Gaussian distribution has mean $0$ and some variance $\sigma_q^2 > 0$, i.e., $\frac{q(\Phi_n)}{\sqrt{n}} \xrightarrow{d} \mathcal{N}(0,\sigma_q^2)$ as $n \to \infty$.
    Thus, we have that 
    \begin{align*}
        \lim_{n \to \infty} \E\left[\frac{\abs{q(\Phi_n)}}{\sqrt{n}}\right] = \sqrt{\frac{2}{\pi}\sigma_q^2} < \infty,
    \end{align*}
	by the standard formula $\E[\abs{X}] = \sqrt{\frac{2}{\pi}}\sigma$ for a Gaussian random variable $X$ with mean $0$ and variance $\sigma^2$.
	In particular, this implies that 
	\begin{align*}
		c \coloneqq \inf_{n \in \Nats} \E\left[\frac{\abs{q(\Phi_n)}}{\sqrt{n}}\right] < \infty, \quad
		C \coloneqq \sup_{n \in \Nats} \E\left[\frac{\abs{q(\Phi_n)}}{\sqrt{n}}\right] < \infty \quad \text{and} \quad
		c' \coloneqq \liminf_{n \to \infty} \E\left[\frac{\abs{q(\Phi_n)}}{\sqrt{n}}\right] < \infty.
	\end{align*}
	Thus, we have $c\sqrt{n} \leq \E[\abs{q(\Phi_n)}] \leq C\sqrt{n}$ for all $n$.
	Note that $c' > 0$ since the limiting Gaussian distribution is non-degenerate since $q$ is not $\muGP^{\mathcal{D}}$-tame.
	It follows that $c'\sqrt{n} \leq \E[\abs{q(\Phi_n)}]$ for all sufficiently large $n$.
\end{proof}
\begin{remark}\label{rmk:autotame}
	Note that if $\mathcal{D}$ is autonomously exhaustive, 
	any Hofer-Lipschitz homogeneous quasimorphism is not $\muGP^{\mathcal{D}}$-tame.
	This can be seen by noting that for any real number there is some Hofer-ball where the quasimorphism takes values larger than that number.
	By~\cite[Theorem 1.13]{dawid-2025}, any Hofer ball is contained in the support of the law of the random walk for some sufficiently large number of steps, which implies that the quasimorphism is not $\muGP^{\mathcal{D}}$-tame.
\end{remark}
We conclude this section by connecting the above result to the growth of the stable commutator length.
Recall that by Bavard duality, the stable commutator length of an element $g$ is related to the supremum of the values of homogeneous quasimorphisms on $g$.
Thus, we can use Theorem~\ref{thm:abs-q-growth} to obtain a growth estimate for the expected value of the stable commutator length under the random walk induced by a centered autonomous law-defining datum if there exists a Hofer-Lipschitz homogeneous quasimorphism that is not $\muGP^{\mathcal{D}}$-tame.
\begin{corollary}\label{cor:scl-growth}
	Let $(M,\omega)$ be a closed symplectic manifold and $\mathcal{D}$ a centered autonomous law-defining datum.
	Assume that there exists a Hofer-Lipschitz homogeneous quasimorphism on $\Ham(M,\omega)$ that is not $\muGP^{\mathcal{D}}$-tame.
	Then there exists an $N \in \Nats$ such that for all $n \geq N$, we have
	$\E[\scl(\Phi_n)] \geq c\sqrt{n}$ for some constant $c > 0$, where $\Phi_n$ is the random walk induced by $\mathcal{D}$.
\end{corollary}
\begin{proof}
	The fact that $\scl(\Phi_n)$ is integrable follows from Corollary~\ref{cor:scl-gaussian-tail}.
	By Bavard duality, we have that $\scl(\Phi_n) \geq \frac{1}{2D(q)}\abs{q(\Phi_n)}$.
	Now Theorem~\ref{thm:abs-q-growth} implies that there exists an $N \in \Nats$ such that for all $n \geq N$, we have $\E[\abs{q(\Phi_n)}] \geq c\sqrt{n}$ for some constant $c > 0$.
	Thus, for such an $n \geq N$, we have 
	\begin{equation*}
	\E[\scl(\Phi_n)] \geq \frac{c}{2D(q)}\sqrt{n},
	\end{equation*}
	which completes the proof.
\end{proof}
\section{The Hofer geometry of random walks}
In the following section we will conclude by deducing some results on the Hofer geometry (and $L^p$-geometry) of $\Ham(M,\omega)$ as seen 
from the viewpoint of random walks.
In particular, we focus on the following question:
how does $\E[\norm{\Phi_n}_{\Hof}]$ behave as $n \to \infty$?
If $\E[\norm{\Phi_n}_{\Hof}]$ grows linearly in $n$, this would suggest that the large-scale geometry of $\Ham(M, \omega)$ is dominated by hyperbolic behavior, while growth at rate $\sqrt{n}$ would suggest that the large-scale geometry is dominated by flat behavior.
Slower growth would suggest positive curvature behavior.
\subsection{The lower bound}
We will now show
that under certain assumptions we can define a lower bound on the expected value of the Hofer norm of the random walk.
Here we will use the result from Section~\ref{sec:clt}.
The goal is to show Theorem~\ref{thm:lower-bound-Ham}, which we recall below:
\lowerBoundOnHofer*
\begin{remark}
    Using the Entov-Polterovich quasimorphism and its various generalizations this theorem applies to a large class of symplectic manifolds with 
    sufficiently rich quantum cohomology. Link spectral invariants offer a further large class of quasimorphisms that can be used to apply the proposition in the two-dimensional setting.
\end{remark}
\begin{proof}
    This proof starts out similar to that of Theorem~\ref{thm:abs-q-growth}.
    Again, we note that by Theorem~\ref{thm:q-clt}, 
    the distribution of the random variable $\frac{q(\Phi_n)}{\sqrt{n}}$ converges to a (non-degenerate) Gaussian distribution as $n \to \infty$. 
    As in the proof of Theorem~\ref{thm:abs-q-growth}, we have that $\E[q(\Phi_n)] = 0$ for all $n$ and thus $\frac{q(\Phi_n)}{\sqrt{n}}$ converges in distribution to a Gaussian distribution with mean $0$ and some variance $\sigma_q^2 > 0$.
    As before,
    \begin{align*}
        \lim_{n \to \infty} \E\left[\frac{\abs{q(\Phi_n)}}{\sqrt{n}}\right] = \sqrt{\frac{2}{\pi}\sigma_q^2}.
    \end{align*}
    Note that we have the following, by the Hofer-Lipschitz property of $q$, we have $\abs{q(\Phi_n)} \leq L_q \cdot d_{\Hof}(\id, \Phi_n)$ where $L_q$ is the Lipschitz constant of $q$.
    Now let $\eps > 0$ be arbitrarily small. Then 
    there is some $N \in \Nats$ such that for all $n \geq N$ we have $\E\left[\frac{\abs{q(\Phi_n)}}{\sqrt{n}}\right] \geq \sqrt{\frac{2}{\pi}\sigma_q^2} - \eps$.
    Of course, $\eps$ is sufficiently small for this quantity to be positive. 
    Now, for any $n \geq N$ we have that 
    \begin{align*}
        \E[d_{\Hof}(\id, \Phi_n)] \geq \frac{1}{L_q} \cdot \E[\abs{q(\Phi_n)}] = \frac{1}{L_q} \sqrt{n} \cdot \E\left[\frac{\abs{q(\Phi_n)}}{\sqrt{n}}\right]
        \geq \frac{\sqrt{\frac{2}{\pi}\sigma_q^2} - \eps}{L_q} \sqrt{n}.
    \end{align*}
    Now we are almost done.
    Given our assumptions on $\mathcal{D}$, we have that the 
    support of $(\muGP^{\mathcal{D}})^{*k}$ is never concentrated at the identity for any $k$.
    Thus, $\E[d_{\Hof}(\id, \Phi_n)]$ is positive for all $n$.
    Therefore, we obtain 
    \begin{align*}
        c \coloneqq \min\left\{\frac{\sqrt{\frac{2}{\pi}\sigma_q^2} - \eps}{L_q}, \E[d_{\Hof}(\id, \Phi_1)], \frac{\E[d_{\Hof}(\id, \Phi_2)]}{\sqrt{2}}, \dots, \frac{\E[d_{\Hof}(\id, \Phi_N)]}{\sqrt{N}}\right\} > 0.
    \end{align*}
    It follows from the construction that 
    \begin{align*}
        \E[d_{\Hof}(\id, \Phi_n)] \geq c\sqrt{n}
    \end{align*}
    for all $n \in \Nats$. This shows the first part of the proposition.

    It remains to show that almost all sample paths leave any bounded region at some point in time.
    Consider the set of all $\theta \in \Omega$ such that 
    the associated sample path leaves any bounded region, i.e.\
    \begin{equation*} 
        \mathcal{U}_0 \coloneqq \left\{ \theta \in \Omega \mid \limsup_{n \to \infty} d_{\Hof}(\id, \Phi_n(\theta)) = +\infty \right\}.
    \end{equation*} 
    Note that 
    \begin{align*}
        \mathcal{U}_1 \coloneqq \left\{ \theta \in \Omega \mid \limsup_{n \to \infty} \abs{q(\Phi_n(\theta))} = +\infty \right\}
    \end{align*}
    satisfies $\mathcal{U}_1 \subset \mathcal{U}_0$ by the Hofer-Lipschitz property of $q$.
    Now by the law of the iterated logarithm from Theorem~\ref{thm:q-clt}, we have that $\Prob(\mathcal{U}_1) = 1$,
    and thus $\Prob(\mathcal{U}_0) = 1$ which completes the proof.
\end{proof}
As stated before, there is a large class of symplectic manifolds for which the above theorem applies, 
although the existence of a Hofer-Lipschitz homogeneous quasimorphism that is not $\muGP^{\mathcal{D}}$-tame is a non-trivial condition to verify in general. However, many such quasimorphisms are known to exist.
For the convenience of the reader, we formulate the following corollary:
\begin{corollary}\label{cor:lower-bound-Ham}
    Let $(M,\omega)$ be one of the following symplectic manifolds:
    \begin{enumerate}
        \item the sphere $S^2$ with any area form $\omega$;
        \item the $n$-fold product $S^2 \times \cdots \times S^2$ with the split symplectic form given by $\omega^{\oplus n}$ for any $n \in \Nats$;
        \item the complex projective space $\C P^n$ with the Fubini-Study form;
        \item the product $\C P^{n_1} \times \cdots \times \C P^{n_k}$ with a monotone symplectic form for any $k \in \Nats$ and $n_1, \dots, n_k \in \Nats$.
    \end{enumerate}
    Then for any centered autonomously exhaustive law-defining datum $\mathcal{D}$ on $M$ there exists a constant $c > 0$ such that $\E[d_{\Hof}(\id, \Phi_n)] \geq c\sqrt{n}$ for all $n$, where $\Phi_n$ is the random walk induced by $\mathcal{D}$.
\end{corollary}
\begin{proof}
    This follows directly by applying Theorem~\ref{thm:lower-bound-Ham} to the quasimorphism constructed by Entov and Polterovich, see Theorem~\ref{thm:entov-polterovich} for the $S^2$ and $\C P^n$ cases.
    In the remaining cases, we use the generalization of the Entov-Polterovich quasimorphism from Theorem~\ref{thm:toric-Entov-Polterovich}
    and note that it descends to $\Ham(M,\omega)$ in these cases, see~\cite{branson-2011}.
    In all cases, this quasimorphism is Hofer-Lipschitz and non-trivial on $\Ham(M,\omega)$.
    By Remark~\ref{rmk:autotame}, it is therefore not $\muGP^{\mathcal{D}}$-tame since it is non-trivial and $\mathcal{D}$ is autonomously exhaustive.
    Thus, the assumptions of Theorem~\ref{thm:lower-bound-Ham} are satisfied, and we obtain the desired conclusion.
\end{proof}
One should note that it is possible to extend this picture to the universal cover $\widetilde{\Ham}(M,\omega)$ of $\Ham(M,\omega)$, and to obtain a similar lower bound on the drift of the random walk on $\widetilde{\Ham}(M,\omega)$ in a more general setting.
For this, recall from~\cite{dawid-2025} that 
any law-defining datum $\mathcal{D}$ on $M$ induces a measure on $\widetilde{\Ham}(M,\omega)$, which we denote by $\widetilde{\muGP}^{\mathcal{D}}$.
If $\mathcal{D}$ is autonomous, this still holds true, and we can consider the random walk on $\widetilde{\Ham}(M,\omega)$ induced by $\widetilde{\muGP}^{\mathcal{D}}$.
For this, let $\{\tilde{\phi}_n\}_{n \in \Nats}$ be independent and $\widetilde{\muGP}^{\mathcal{D}}$-distributed random variables
and let 
\begin{equation*}
    \tilde{\Phi}_n \coloneqq \tilde{\phi}_1 \cdots \tilde{\phi}_n
\end{equation*}
be the associated random walk on $\widetilde{\Ham}(M,\omega)$.
Then we have the following corollary:
\begin{corollary}\label{cor:lower-bound-universal-cover}
    Let $(M,\omega)$ be a closed toric symplectic manifold.
    Then for any centered autonomously exhaustive law-defining datum $\mathcal{D}$ on $M$ there exists a constant $c > 0$ such that $\E[d_{\Hof}(\id, \tilde{\Phi}_n)] \geq c\sqrt{n}$ for all $n$, where $\tilde{\Phi}_n$ is the random walk on $\widetilde{\Ham}(M,\omega)$ induced by $\mathcal{D}$.
\end{corollary}
\begin{proof}
    The Hofer metric on $\widetilde{\Ham}(M,\omega)$ is bi-invariant, so that 
    we can apply the central limit theorem (Theorem~\ref{thm:q-clt}) to the quasimorphism constructed by Usher and Fukaya--Oh--Ohta--Ono, see Theorem~\ref{thm:toric-Entov-Polterovich}.
    From there the proof is identical to that of Corollary~\ref{cor:lower-bound-Ham}.
\end{proof}
\subsection{Some lower bounds on the $L^p$-geometry of the random walk}
In the previous section we have seen that the central limit theorem for quasimorphisms can be used to obtain a lower bound on the growth of $\E[d_{\Hof}(\id, \Phi_n)]$.
Note that Theorem~\ref{thm:q-clt} is not specific to the Hofer metric, but rather applies to any right-invariant metric on $\Ham(M,\omega)$.
Recall from Section~\ref{sec:prelims:lp} that there is a family of right-invariant metrics $d_{p}$ on $\Ham(M,\omega)$ for $p \in [1,\infty)$.
These $L^p$-metrics are defined by taking an infimum over the average $L^p$-norm of the Hamiltonian vector field generating a path connecting two Hamiltonian diffeomorphisms. In the following, we assume that the metric $d_p$ is defined with respect to a Riemannian metric $g$ on $M$ that is induced by the almost-complex structure $J$ of the relevant \ldd .

Recall from Section~\ref{sec:prelims:lp} that the Gambaudo--Ghys quasimorphisms on symplectic surfaces are Lipschitz with respect to all $L^p$-metrics.
Thus, by applying Theorem~\ref{thm:q-clt} to the Gambaudo--Ghys quasimorphisms, we can obtain a lower bound on the growth of the expected $L^p$-distance of the random walk $\Phi_n$ from the identity.
Before we state the result, note the following useful lemma:
\begin{lemma}
    Let $(M,\omega)$ be a closed symplectic manifold and $J$ a compatible almost-complex structure on $M$.
    Then for any $p \in [1,\infty)$ we have 
    \begin{equation*}
        \ell_p(\{\phi_H^t\}_{t \in [0,1]}) \leq \norm{H}_{C^1}
    \end{equation*}
    for any $H \in C^\infty([0,1] \times M)$.
\end{lemma}
\begin{proof}
    First, note that $X_H(t,\cdot) = J \nabla H_t$, where $H_t(\cdot) = H(t,\cdot)$ and $\nabla$ is the gradient with respect to the Riemannian metric $g$ induced by $J$.
    Then we have that 
    \begin{equation*}
        \norm{X_H(t,\cdot)}_{g} = \omega(X_H(t,\cdot), J X_H(t,\cdot)) = \omega(J \nabla H_t, J^2 \nabla H_t) = \omega(\nabla H_t, J \nabla H_t) = \norm{\nabla H_t}_{g}
    \end{equation*}
    for any $t \in [0,1]$.
    By integrating over $t \in [0,1]$ we obtain the desired inequality.
    Namely, 
    \begin{align*}
        \ell_p(\{\phi_H^t\}_{t \in [0,1]}) &= \int_0^1 \left(\int_M \frac{1}{\vol(M,\omega)} \norm{X_H(t,\cdot)}_g^p \, \omega^n\right)^{1/p} dt 
                                           = \int_0^1 \left(\int_M \frac{1}{\vol(M,\omega)} \norm{\nabla H_t}_{g}^p \, \omega^n\right)^{1/p} dt \\
                                           &\leq \int_0^1 \left(\int_M \frac{1}{\vol(M,\omega)} \norm{H_t}_{C^1}^p \, \omega^n\right)^{1/p} dt
                                           \leq \norm{H}_{C^1}.
    \end{align*}
\end{proof}
We can now obtain the following lower bound on the growth of the random walk in the $L^p$-metric on $\Ham(\Sigma,\omega)$ for a closed symplectic surface $(\Sigma,\omega)$.
\begin{theorem}
    Let $(\Sigma,\omega)$ be a closed symplectic surface and $\mathcal{D}$ a centered autonomously exhaustive law-defining datum on $\Sigma$.
    Then for any $p \in [1,\infty)$ there exists a constant $c > 0$ such that $\E[d_{p}(\id, \Phi_n)] \geq c\sqrt{n}$ for all $n$, where $\Phi_n$ is the random walk induced by $\mathcal{D}$. Furthermore, almost all sample paths leave any bounded region eventually.
\end{theorem}
\begin{proof}
    We note that the proof of Theorem~\ref{thm:lower-bound-Ham} applies verbatim to any right-invariant metric on $\Ham(M,\omega)$, and in particular to the $L^p$-metric $d_p$. In our case a suitable quasimorphism is given by any Gambaudo--Ghys quasimorphism, which is Lipschitz with respect to $d_p$ by a result of Brandenbursky--Marcinkowski--Shelukhin, see Theorem~\ref{thm:gambaudo-ghys}.
    Thus, we simply need to verify that the Gambaudo--Ghys quasimorphism is not $\muGP^{\mathcal{D}}$-tame for any centered autonomously exhaustive law-defining datum $\mathcal{D}$ on $\Sigma$.
    However, the quasimorphism is non-trivial on $\Ham(\Sigma,\omega)$.
    In particular, by the above lemma, this means (due to the fact that it is Lipschitz with respect to $d_p$) that 
    there exists an entire open ball in the $C^1$-topology on $C^\infty([0,1] \times \Sigma)$ on which the quasimorphism is non-trivial.
    Since $\mathcal{D}$ is autonomously exhaustive, the support of the law of $\Phi_n$ will include some such ball for a sufficiently large $n$.
    Thus, we can conclude that the quasimorphism is not $\muGP^{\mathcal{D}}$-tame for any centered autonomously exhaustive law-defining datum $\mathcal{D}$ on $\Sigma$. This completes the argument.
\end{proof}
\begin{remark}
    By applying the central limit theorem for Lipschitz quasimorphisms (Theorem~\ref{thm:q-clt}) to the Gambaudo--Ghys quasimorphisms, 
    we can also gain some dynamical insight into the behavior of the trajectories of the random walk $\Phi_n$.
    For this, we note that the Gambaudo--Ghys quasimorphisms measure a certain notion of average asymptotic braiding of trajectories of finite configurations of points under the Hamiltonian isotopy.
    We refer the reader to~\cite{gambaudo-ghys-2004} for details.
    Thus, high values of Gambaudo--Ghys quasimorphisms indicate that for some finite configuration of points on the surface, the trajectories of these points under the Hamiltonian isotopy are highly braided. From this perspective, the above theorem implies that for any centered autonomously exhaustive law-defining datum $\mathcal{D}$ on a closed symplectic surface $(\Sigma,\omega)$, almost all sample paths of the random walk $\Phi_n$ will eventually exhibit highly braided behavior for any suitable finite configuration of points on $\Sigma$.
    Furthermore, the $L^1$-metric on $\Ham(\Sigma,\omega)$ can be interpreted as the average length of trajectories of points on $\Sigma$ under the Hamiltonian isotopy.
    Thus, the above theorem also implies that for any centered autonomously exhaustive law-defining datum $\mathcal{D}$ on a closed symplectic surface $(\Sigma,\omega)$, almost all sample paths of the random walk $\Phi_n$ will eventually exhibit trajectories that are arbitrarily long.
\end{remark}
\begin{remark}
    Similarly to Corollary~\ref{cor:scl-growth}, 
    the above theorem implies that the stable commutator length of the random walk $\Phi_n$ grows at least at rate $\sqrt{n}$ asymptotically when $\mathcal{D}$ is a centered autonomously exhaustive law-defining datum on a closed symplectic surface $(\Sigma,\omega)$
    and $\Phi_n$ is the random walk induced by $\mathcal{D}$.
\end{remark}
\subsection{Some upper bounds on the expected Hofer norm of the random walk}
In general, an upper bound on the growth of $\E[d_{\Hof}(\id, \Phi_n)]$ has so far proven elusive 
except for the trivial linear bound.
However, for some special cases, we can obtain a sub-linear upper bound. Specifically, 
when the random walk is constrained to an abelian subgroup of $\Ham(M,\omega)$, we can obtain an upper bound of order $\sqrt{n}$.
\subsubsection{Abelian subgroups of $\Ham(M,\omega)$}
While there are of course no normal abelian subgroups of $\Ham(M,\omega)$, there are many abelian subgroups of $\Ham(M,\omega)$ given by autonomous Hamiltonians and integrable systems as well as various other constructions.
Interestingly, the study of such abelian subgroups is central to the study of $\Ham(M,\omega)$ and its geometry.
The construction of Hofer flats in $\Ham(M,\omega)$ is also in many cases a process that involves (at least implicitly)
the construction of abelian subgroups of $\Ham(M,\omega)$.
It is a classical fact that for finite-dimensional Lie groups, flatness and commutativity are equivalent.
Thus, a first step towards understanding the diffusion of the random walk $\Phi_n$ might be to understand 
it in cases where the random walk is constrained to a (possibly infinite-dimensional) abelian subgroup of $\Ham(M,\omega)$.
It is reasonable to think that by restricting to a law-defining datum $\mathcal{D}$ whose support consists of commuting Hamiltonian diffeomorphisms, we can obtain an upper bound on the growth of $\E[d_{\Hof}(\id, \Phi_n)]$ by utilizing that commutativity.
We start by the following definition:
\begin{definition}
    Let $\mathcal{D}$ be an autonomous law-defining datum on a symplectic manifold $(M,\omega)$.
    Then we say that $\mathcal{D}$ is \emph{commutative}, if for any two 
    independent versions $\randH_1,\randH_2$ of the Gaussian process~\eqref{eq:GP_H} associated with $\mathcal{D}$
    we have that 
    \begin{equation*}
        \Prob\left[\{\randH_1, \randH_2\} = 0\right] = 1,
    \end{equation*}
    where $\{f, g\} \coloneqq \omega(X_{f}, X_{g})$ denotes the Poisson bracket on $M$.
\end{definition}
The following lemma is an easy computation which we will use in the proof of Theorem~\ref{thm:upper-bound-commutative}.
We recall it here for the convenience of the reader.
\begin{lemma}\label{lem:commutative-Hamiltonians}
    Let $H_1,\dots,H_m$ be autonomous Hamiltonian functions on $M$ such that for any $i,j$ we have 
    $\{H_i,H_j\} = 0$.
    Then $H_1 \# \cdots \# H_m = H_1 + \cdots + H_m$.
\end{lemma}
\begin{proof}
    This can be straightforwardly proved by induction.
    Let $m=2$. 
    Note that 
    \begin{align*}
        \frac{d}{dt} H_2(\phi_{H_1}^t(x)) \Big\vert_{t=s} = dH_2(X_{H_1}(\phi_{H_1}^s(x))) = \{H_2,H_1\}(\phi_{H_1}^s(x)) = 0,
    \end{align*}
    which implies that $H_2(\phi_{H_1}^s(x))$ is constant in $s \in \Reals$.
    Then we have that 
    \begin{equation*}
        H_1 \# H_2(t,x) = H_1(x) + (H_2 \circ \phi_{H_1}^{-t})(x) = H_1(x) + H_2(x).
    \end{equation*}
    Now let us formulate the induction step.
    We assume that $H_1 \# \cdots \# H_{m-1} = H_1 + \cdots + H_{m-1}$ holds true.
    Then 
    \begin{align*}
        (H_1 \# \cdots \# H_m)(t,x)
        &= (H_1 \# \cdots \# H_{m-1}) \# H_m(t,x)
        = (H_1 + \cdots + H_{m-1}) \# H_m(t,x) \\
        &= H_1(x) + \cdots + H_{m-1}(x) + (H_m \circ \phi_{H_1 + \cdots + H_{m-1}}^{-t})(x).
    \end{align*}
    Note that 
    \begin{align*}
        \frac{d}{dt} H_m(\phi_{H_1 + \cdots + H_{m-1}}^{-t}(x)) \Big\vert_{t=s} 
        &= dH_m(X_{H_1 + \cdots + H_{m-1}}(\phi_{H_1 + \cdots + H_{m-1}}^{-s}(x))) \\
        &= \sum_{i=1}^{m-1} dH_m(X_{H_i}(\phi_{H_1 + \cdots + H_{m-1}}^{-s}(x))) \\
        &= \sum_{i=1}^{m-1} \{H_m,H_i\}(\phi_{H_1 + \cdots + H_{m-1}}^{-s}(x)) = 0,
    \end{align*}
    for any $s \in \Reals$. Thus, the last term in the previous expression is constant in $t$, and equal to $H_m(x)$,
    which then implies that
    $H_1 \# \cdots \# H_m = H_1 + \cdots + H_m$.
\end{proof}
Of course, this in particular means that 
a random walk generated by a commutative law-defining datum will never leave the set of autonomous 
Hamiltonian diffeomorphisms.
At first glance, one might think this is a restriction.
However, that is not the case.
The following simple lemma shows us that any connected commutative subgroup 
must lie within $\Aut(M,\omega)$.
Thus, if we want to restrict our random walks so such commutative 
subgroups, the notion of a commutative \ldd\ does not
impose any artificial restrictions.
\begin{lemma}
    Let $G < \Ham(M,\omega)$ be a connected abelian subgroup.
    Then $G \subset \Aut(M,\omega)$.
\end{lemma}
\begin{proof}
    Let $\phi \in G$ be arbitrary.
    Since $G$ is connected (and thereby automatically path-connected),
    there exists some path $\{\phi_t\}_{t\in[0,1]} \subset G$ from $\phi_0 = \id$ 
    to $\phi_1 = \phi$.
    This path is generated by some unique normalized Hamiltonian function $H \in C^\infty_0([0,1] \times M)$, i.e., $\phi_H^t = \phi_t$
    for all $t \in [0,1]$.
    Note that since $G$ is commutative, we have $\phi_t = \phi_t \circ \phi_s \circ \phi_s^{-1} =\phi_s^{-1}\circ \phi_t \circ \phi_s =\phi_s\circ \phi_t \circ \phi_s^{-1}$.
    By differentiating with respect to $t$, we obtain that 
    $(\phi_s)_*X_H(t,\cdot) = X_H(t,\cdot)$ for all $s,t \in [0,1]$.
    By now differentiating with respect to $s$, we obtain 
    that $[X_H(t,\cdot), X_H(s,\cdot)] = 0$ for all $s,t \in [0,1]$.
    Since $[X_f,X_g] = -X_{\{f,g\}}$, this implies 
    that $\{X_H(t,\cdot),X_H(s,\cdot)\} = 0$.
    The conclusion is now essentially implied by Lemma~\ref{lem:commutative-Hamiltonians}.
    In particular, 
    we can define 
    \begin{equation*}
        \bar{H}(x) = \int_0^1 H(t,x) dt.
    \end{equation*}
    Applying Lemma~\ref{lem:commutative-Hamiltonians} to Riemann sums approximating this integral 
    together with $\{X_H(t,\cdot),X_H(s,\cdot)\} = 0$
    for any $s,t \in [0,1]$ implies that 
    $\phi^1_{\bar{H}} = \phi_H^1 = \phi \in \Aut(M,\omega)$.
    Since $\phi \in G$ was arbitrary, this concludes the proof.
\end{proof}
With this context laid out, we 
can proceed to the main objective of the section, namely to proving Theorem~\ref{thm:upper-bound-commutative}.
\begin{proof}[Proof of Theorem~\ref{thm:upper-bound-commutative}]
    Let $\randH_1,\randH_2, \dots$ be independent versions of the Gaussian process associated with $\mathcal{D}$.
    Note that by~\cite[Theorem 1.9]{dawid-2025} the induced law on $\Aut(M,\omega)$ is invariant under inversion, since $\mathcal{D}$ is centered.
    Thus, we have that the law of $\Phi_k$ is the same as the law of $\phi_{\randH_1}^1 \circ \cdots \circ \phi_{\randH_k}^1$.
    In particular, 
    \begin{equation*}
        \E\left[d_{\Hof}(\id, \Phi_k)\right] = \E\left[ \norm*{\phi_{\randH_1}^1 \circ \cdots \circ \phi_{\randH_k}^1}_{\Hof}\right].
    \end{equation*}
    Furthermore, 
    \begin{align*}
        \norm*{\phi_{\randH_1}^1 \circ \cdots \circ \phi_{\randH_k}^1}_{\Hof} \leq 2 \int_0^1 \norm*{(\randH_1 \# \cdots \# \randH_k)(t,\cdot)}_{\infty} dt 
    \end{align*}
    by definition of the Hofer norm.
    Now, note that $\{\randH_i, \randH_j\} = 0$ holds almost-surely for any $i,j$ by the commutativity assumption on $\mathcal{D}$.
    Thus, we can almost-surely apply Lemma~\ref{lem:commutative-Hamiltonians} to obtain that 
    \begin{align*}
        \int_0^1 \norm*{(\randH_1 \# \cdots \# \randH_k)(t,\cdot)}_{\infty} dt 
        = \norm*{\randH_1 + \cdots + \randH_k}_{\infty}
    \end{align*}
    holds almost-surely.
    Recall that the process $\randH_i$ can be written as 
        $\randH_i = \sum_{j=1}^{\infty} w_j \cdot Z_{i,j} \cdot e_j$,
    where $w_j$ are the weights, $e_j$ are the eigenfunctions of the Laplacian, and $Z_{i,j}$ are Gaussian variables as specified by the \ldd\ $\mathcal{D}$.
    By construction, $Z_{i,j}$ and $Z_{i',j}$ are independent for any $i,i' \in \Nats$.
    Let $\sigma_j^2$ be the variance of $Z_{i,j}$ for any $i \in \Nats$.
    Recall from~\cite[Definition 3.1]{dawid-2025} that $\sup_{j} \sigma_j^2 < \infty$.
    To conclude the proof, we need another fact from elementary probability theory: if $Y_1,\dots,Y_n$ are independent identically distributed centered Gaussian variables with variance $\sigma^2$, then \begin{equation*}\E[\abs{Y_1 + \cdots + Y_n}] = \sqrt{\frac{2}{\pi} n \sigma^2}.\end{equation*}
    It follows that 
    \begin{align*}
        \E\left[ \norm*{\randH_1 + \cdots + \randH_k}_{\infty} \right] &\leq \sum_{j=1}^{\infty} w_j \cdot \E\left[\abs*{\sum_{i=1}^k Z_{i,j}}\right] \cdot \norm{e_j}_\infty = \sum_{j=1}^{\infty} w_j\norm{e_j}_\infty \cdot \sqrt{\frac{2}{\pi} k \sigma_j^2} \\
        &\leq \underbrace{\left(\sup_{j \in \Nats} \sqrt{\frac{2}{\pi}\sigma^2_j} \cdot \sum_{j=1}^{\infty} w_j \norm{e_j}_\infty \right)}_{\eqqcolon C} \cdot \sqrt{k}.
    \end{align*}
    Note that $C$ is finite since $\sup_{j} \sigma_j^2 < \infty$ and $\sum_{j=1}^{\infty} w_j \norm{e_j}_\infty < \infty$. The latter follows from the fact that the weights $w_j$ decay exponentially in $j$ and $\norm{e_j}_\infty$ grows at most polynomially in $j$ by standard estimates on the growth of eigenfunctions of the Laplacian.
    By combining the above estimates, we obtain that $\E[d_{\Hof}(\id, \Phi_k)] \leq 2C \sqrt{k}$ for all $k \in \Nats$, which completes the proof.
\end{proof}
\subsubsection{Toy example: the torus}
In the last section we obtained an abstract result.
However, the question remains: are there any (non-trivial) examples of law-defining data $\mathcal{D}$ for which the assumptions of Theorem~\ref{thm:upper-bound-commutative} are satisfied?

The easiest example of this is the torus $\T^{2n}$, equipped with the standard symplectic form $\omega$ and the standard complex structure $J$, inducing the standard flat Riemannian metric $g$.
Then, as a Riemannian manifold, $\T^{2n} \cong \T^n \times \T^n$, where each $\T^n$ is equipped with the standard flat Riemannian metric.
We call the projection onto the first factor $\pi\colon \T^{2n} \to \T^n$.
It is easy to see that for any eigenfunction $f$ of the Laplacian $\Laplace_{\T^n}$ on $\T^n$, the function $f \circ \pi$ is an eigenfunction of the Laplacian $\Laplace_{\T^{2n}}$ on $\T^{2n}$ with the same eigenvalue.
Thus, we can construct a law-defining datum $\mathcal{D}$ on $\T^{2n}$ whose support consists of Hamiltonian diffeomorphisms that are generated by functions lifted from $\T^n$.
For the resulting random walk we can establish a square-root order upper bound on the growth of $\E[d_{\Hof}(\id, \Phi_n)]$
by applying Theorem~\ref{thm:upper-bound-commutative} since the Poisson bracket of any two functions lifted from $\T^n$ vanishes.

To formalize this, let us start with the following definition:
\begin{definition}
    Let $\mathcal{D}$ be a law-defining datum on $\T^{2n}$ equipped with the standard symplectic form $\omega$ and the standard complex structure $J$.
    We say that $\mathcal{D}$ is \emph{lifted from the base} if the coefficient process $Z_n$ associated with any eigenfunction $e_n$ that is not of the form $f \circ \pi$ for an eigenfunction $f$ of $\Laplace_{\T^n}$ is almost-surely zero.
    We further say that $\mathcal{D}$ is \emph{exhaustively lifted from the base} if additionally we have that the support of $\muGP^{\mathcal{D}}$ contains all Hamiltonian diffeomorphisms generated by functions lifted from $\T^n$, i.e., if $\{\phi_{f \circ \pi}^1 \mid f \in C^\infty(\T^n)\} \subset \supp(\muGP^{\mathcal{D}})$.
\end{definition}
With this definition in hand, we can state and prove the following proposition:
\begin{proposition}
    Let $\mathcal{D}$ be an autonomous centered law-defining datum on $\T^{2n}$ that is lifted from the base.
    Denote by $\Phi_k$ the associated random walk on $\Ham(\T^{2n},\omega)$.
    Then there exists a constant $C > 0$ such that $\E[d_{\Hof}(\id, \Phi_k)] \leq C\sqrt{k}$ for all $k \in \Nats$.
\end{proposition}
\begin{proof}
    By Theorem~\ref{thm:upper-bound-commutative}, it suffices to show that $\mathcal{D}$ is commutative.
    Let $\randH_1,\randH_2$ be independent versions of the Gaussian process associated with $\mathcal{D}$.
    Let $e_n$ be an eigenbasis of the Laplacian $\Laplace_{\T^{n}}$ on $\T^{n}$.
    Then we know that there exist two independent\footnote{We mean that any element of the first family is independent from any element of the second family.} families of Gaussian variables $Z_{i,n}$ for $i=1,2$ and $n \in \Nats$ such that $\randH_i = \sum_{n=1}^{\infty} w_n \cdot Z_{i,n} \cdot (e_n \circ \pi)$ for $i=1,2$ by the definition of being lifted from the base.
    An easy computation then shows that 
    \begin{align*}
        \{\randH_1, \randH_2\} = \sum_{n,m=1}^{\infty} w_n w_m Z_{1,n} Z_{2,m} \cdot \{e_n \circ \pi, e_m \circ \pi\} = 0.
    \end{align*}
    Here we first use the bilinearity of the Poisson bracket and then the fact that the Poisson bracket of $\{e_n \circ \pi, e_m \circ \pi\}$ vanishes for any $n,m$ since the Hamiltonian vector field of a function lifted from $\T^n$ is tangent to the fibers of $\pi$ and these fibers are Lagrangian.
    In particular, $\mathcal{D}$ is commutative, which completes the proof.
\end{proof}
\subsubsection{Toric varieties}
The torus is itself famously not a toric variety.
However, the most crucial idea in the above example is that the Hamiltonian vector field of a function lifted from the base is tangent to the fibers of the projection $\pi\colon \T^{2n} \to \T^n$ and these fibers are Lagrangian.
Thus, we can hope to extend the above example to more general symplectic manifolds that admit Lagrangian fibrations.
In general, we wish to allow for singularities in the fibration, and thus we cannot expect to be able to lift eigenfunctions from the base to the total space in general.
However, a large class of manifolds that admit Lagrangian fibrations are toric varieties.
These additionally have the advantage that by studying eigenfunctions which are invariant under the torus action, 
we can easily generalize the idea from the previous section.
In a general Lagrangian fibration, this might present a problem.

Recall that a toric variety is a closed Kähler manifold $(M^{2n},\omega,I)$ equipped with an effective Hamiltonian action of $\T^n$.
Associated with this action, there is a moment map $\mu\colon M \to \Reals^n$ whose image is a convex polytope $\Delta_{D}$.
This polytope is called the Delzant polytope or moment polytope associated with $M$.
The following is a standard fact, see e.g.~\cite{guillemin-1994} or~\cite{schwarz-1975} for a more general case:
\begin{lemma}
    Smooth functions on the closure of $\Delta_{D}$ are in one-to-one correspondence with smooth functions on $M$ that are invariant under the torus action.
\end{lemma} 
In particular, we can construct a law-defining datum $\mathcal{D}$ on $M$ whose support consists of Hamiltonian diffeomorphisms that are generated by functions invariant under the torus action.
For the resulting random walk we can establish a square-root order upper bound on the growth of $\E[d_{\Hof}(\id, \Phi_n)]$
by applying Theorem~\ref{thm:upper-bound-commutative} since the Poisson bracket of any two functions invariant under the torus action vanishes.
\begin{definition}
    Let $\mathcal{D}$ be a law-defining datum on a toric variety $(M,\omega,I)$ with moment map $\mu\colon M \to \Reals^n$ and associated Delzant polytope $\Delta_{D}$.
    We say that $\mathcal{D}$ is \emph{toric-compatible} if the coefficient process $Z_n$ associated with any eigenfunction $e_n$ that is not of the form $f \circ \mu$ for a smooth function $f$ on $\Delta_{D}$ is almost-surely zero.
\end{definition}
\begin{remark}\label{rmk:toric-compatible-existence}
    Note that such law-defining data always exist. Since $(M, \omega, I)$ is a toric variety, the torus action is by isometries.
    Thus, we can take any eigenfunction of the Laplacian associated with the metric $g(\cdot,\cdot) = \omega(\cdot, I\cdot)$ and average it over the torus action to obtain a smooth eigenfunction (of the same eigenvalue) that is invariant under the torus action.
\end{remark}
We then obtain the following proposition:
\begin{proposition}
    Let $\mathcal{D}$ be a centered autonomous toric-compatible law-defining datum on a toric variety $(M,\omega,I)$.
    Denote by $\Phi_k$ the associated random walk on $\Ham(M,\omega)$.
    Then there exists a constant $C > 0$ such that $\E[d_{\Hof}(\id, \Phi_k)] \leq C\sqrt{k}$ for all $k \in \Nats$.    
\end{proposition}
\begin{proof}
    By Theorem~\ref{thm:upper-bound-commutative}, it suffices to show that $\mathcal{D}$ is commutative.
    Let $\randH_1,\randH_2$ be independent versions of the Gaussian process associated with $\mathcal{D}$.
    Then we have that 
    \begin{align*}
        \{\randH_1, \randH_2\} = \sum_{n,m=1}^{\infty} w_n w_m Z_{1,n} Z_{2,m} \cdot \{e_n \circ \mu, e_m \circ \mu\} = 0,
    \end{align*}
    where $e_n \circ \mu$ and $e_m \circ \mu$ are eigenfunctions of the Laplacian on $M$ that are invariant under the torus action.
    The above equality holds almost-surely since the Poisson bracket of any two functions invariant under the torus action vanishes.
    In particular, $\mathcal{D}$ is commutative, which completes the proof.
\end{proof}
\begin{remark}
    The attentive reader might have noticed that 
    both in the case of the torus and in the case of toric varieties, 
    the crucial property that we use to obtain that a certain \ldd\ is commutative is that the manifold admits a Lagrangian fibration.
    Thus, one might ask: can we extend the above result to more general symplectic manifolds that admit Lagrangian fibrations?
    The answer is presumably yes, but some care has to be taken.
    We generally wish to allow Lagrangian fibrations with singularities, and the presence of these singular fibers can cause problems for the above argument.
    To argue via eigenfunctions, we would want the base of the fibration to be endowed with a compatible 
    metric and for the Lagrangian fibration to be a Riemannian submersion.
    Additionally, 
    to lift the eigenfunctions from the base to the total space, we would need the fibers 
    to be minimal submanifolds of the total space, which is a very strong condition to impose on the fibration. See~\cite{watson-1973} for more details.
    Indeed, this essentially forces us to only consider special Lagrangian fibrations. 
    In the absence of these conditions, one could still hope to obtain a similar result by taking some Gaussian process on the base 
    that is guaranteed to vanish near the singular points and then lift it to the total space, but this
    seems less natural.
    Given the lack of interesting examples in this class, we have not pursued this direction further.
\end{remark}
We can now use the above proposition to obtain ``probabilistically flat'' subspaces of $\Ham(S^{2},\omega)$.
To have such a subspace, we wish to use a toric-compatible law-defining datum on $S^{2}$.
It remains to ensure that such a law-defining datum is non-trivial.
We say that a law-defining datum $\mathcal{D}$ on $S^{2}$ is non-trivial if the support of $\muGP^{\mathcal{D}}$ contains some non-identity Hamiltonian diffeomorphism.
Note that when we endow $S^{2}$ with the standard complex structure, the eigenfunctions of the Laplacian are given by the spherical harmonics.
It is easy to see that there are non-trivial toric-compatible law-defining data on $S^{2}$, e.g.\ by taking a law-defining datum where the coefficient process associated with (some) zonal harmonics are non-trivial and the coefficient processes associated with all other eigenfunctions vanish.
We now obtain the following corollary:
\begin{corollary}\label{cor:exact-growth-S2}
    Let $\mathcal{D}$ be a non-trivial toric-compatible autonomous centered law-defining datum on $S^{2}$.
    Denote by $\Phi_n$ the associated random walk on $\Ham(S^{2},\omega)$.
    Then there exist constants $c,C > 0$ such that $c\sqrt{n} \leq \E[d_{\Hof}(\id, \Phi_n)] \leq C\sqrt{n}$ for all $n$.
\end{corollary}
\begin{proof}
    The upper bound follows from the previous proposition.
    For the lower bound, we unfortunately cannot apply Corollary~\ref{cor:lower-bound-Ham} since $\mathcal{D}$ is not autonomously exhaustive.
    However, we can still apply the same idea, i.e., to apply Theorem~\ref{thm:lower-bound-Ham}.
    Note that the quasimorphisms constructed by Cristofaro-Gardiner, Humili\`ere, Mak, Seyfaddini and Smith in~\cite{cristofaro-gardiner-humiliere-mak-seyfaddini-smith-2022} satisfy a Lagrangian control property, see Section~\ref{sec:prelims:floer}.

    Since $\mathcal{D}$ is non-trivial, the associated Gaussian process $\randH$ is non-trivial as well.
    Furthermore, since $\mathcal{D}$ is toric-compatible, $\randH$ is invariant under the torus action, in this case rotations of the sphere around the $z$-axis.
    Let $h\colon S^{2} \to \Reals$ be the height function.
    Let $z_0 \in [-1,1]$ be chosen such that $\randH$ does not vanish almost-surely at the level set $h^{-1}(z_0)$.
    Then the level set $h^{-1}(z_0)$ is a circle that is invariant under the torus action, and thus it is a Lagrangian submanifold of $S^{2}$.
    We can use this fact in place of the exhaustiveness assumption to obtain that the quasimorphism is not $\muGP^{\mathcal{D}}$-tame, which then allows us to apply Theorem~\ref{thm:lower-bound-Ham} to obtain the desired lower bound.
    Let us take a monotone link $\ul{L} = L_1 \sqcup \cdots \sqcup L_k$ in $S^{2}$ such that it contains $h^{-1}(z_0)$ as a component, and let $\mu_{\ul{L}}$ be the associated quasimorphism.
    Without loss of generality, we can assume that $h^{-1}(z_0)$ is $L_1$.
    Again, without loss of generality, we can assume that $L_i = h^{-1}(z_i)$ for some $z_i \in [-1,1]$ for all $i=1,\dots,k-1$.
    Note that $\randH$ will take a single value on $L_1$, since $\mathcal{D}$ is toric-compatible and $L_1$ is invariant under the torus action.
    Then by the Lagrangian control property of $\mu_{\ul{L}}$ we have that 
    \begin{equation*}
        \mu_{\ul{L}}(\phi_{\randH}^1) = \frac{1}{k} \left( \randH(h^{-1}(z_0)) + 
        \randH(h^{-1}(z_1)) + \cdots + \randH(h^{-1}(z_{k-1})) \right).
    \end{equation*}
    See Figure~\ref{fig:lagrangian-control} for a visualization.
    Since we assumed that $\randH$ does not vanish almost-surely, we have that $\mu_{\ul{L}}(\phi_{\randH}^1)$ does not vanish almost-surely as well (after possibly adding some additional components to the link $\ul{L}$).
    Thus, $\mu_{\ul{L}}$ is not $\muGP^{\mathcal{D}}$-tame, and we can apply Theorem~\ref{thm:lower-bound-Ham} to obtain the desired lower bound.
    Thus, by the same argument as in the proof of Corollary~\ref{cor:lower-bound-Ham} we obtain that $\E[d_{\Hof}(\id, \Phi_n)] \geq c\sqrt{n}$ for some $c > 0$ and all $n$, which completes the proof.
\end{proof}
\begin{figure}[h]
    \centering
    \includegraphics[width=0.5\textwidth]{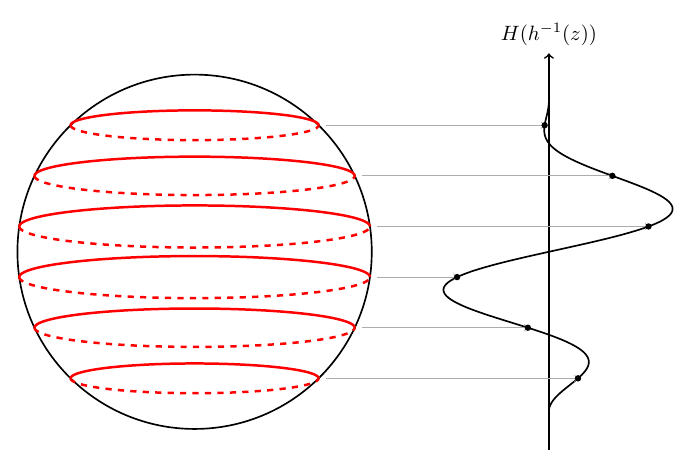}
    \caption{A visualization of the Lagrangian control property of the quasimorphism $\mu_{\ul{L}}$ associated with a monotone link $\ul{L} = L_1 \sqcup \cdots \sqcup L_k$ in $S^{2}$.
    The quasimorphism $\mu_{\ul{L}}$ is controlled by the values of the Hamiltonian function on the components of the link.}
    \label{fig:lagrangian-control}
\end{figure}
It is worth noting that the above corollary can be seen as a probabilistic version of the quasi-flat constructed by Polterovich and Shelukhin in~\cite{polterovich-shelukhin-2023}. 
See also Figure~\ref{fig:lagrangian-control}.
In some sense, we are looking at a random walk that is restricted to this quasi-flat, and we can see that it behaves as expected on a flat space.

The corollary above only applies to $S^2$, as this is the only surface that (1) is known to admit a non-trivial Hofer-Lipschitz quasimorphism, and (2) is a toric variety (when viewed as $S^2 \cong \C P^1$). 
The Lagrangian control property of the quasimorphisms coming from link spectral invariants is another crucial ingredient in the proof.
For a general $\C P^n$, computing the value of the Entov-Polterovich quasimorphism is more complicated.
However, when restricting to Hamiltonians that are invariant under the torus action, 
we can use a result of Biran, Entov, and Polterovich~\cite{biran-entov-polterovich-2004}.
By slightly strengthening the assumptions on the law-defining datum $\mathcal{D}$, we can then obtain a similar result to Corollary~\ref{cor:exact-growth-S2} for all complex projective spaces $\C P^n$ for $n \geq 1$ and their monotone products.

We first quickly recall the relevant definitions and results from~\cite{biran-entov-polterovich-2004}.
Let $\C P^n$ be equipped with the Fubini-Study form $\omega_{FS}$. We normalize $\omega_{FS}$ so that the integral of $\omega_{FS}$ over
the complex projective line is 1. In particular, $\vol(\C P^n) = \int_{\C P^n} \omega_{FS}^{n} = 1$.  
Recall that the $n$-torus $\T^n$ acts on $\C P^n$ by 
$[z_0 :\ldots : z_n] \mapsto
[z_0 : e^{2\pi i s_1} z_1 :\ldots : e^{2\pi i s_n} z_n]$ for $(s_1,\ldots, s_n)\in \T^n$.
A moment map $\mu: \C P^n\to\R^n$ for the action is given by the
formula:
\[\mu([z_0 :\ldots : z_n]) = \left(\frac{|z_1|^2}{|z_0|^2+ \ldots +|z_n|^2}, \ldots , \frac{|z_n|^2}{|z_0|^2+\ldots+|z_n|^2} \right). \]
The moment polytope $\Delta$ is the image of $\mu$ and is given by
$\Delta= \{ p\in \R^n\ | \ p_1 +\ldots+p_n \leq 1\, ; \, p_i\geq 0, \,
i=1,\ldots, n \}$.
The fiber of this map over a point $p\in\Delta$ is an isotropic torus.
In particular, let $
\bar{p} = \left(\frac{1}{n+1},\ldots, \frac{1}{n+1}\right)$
be the barycenter of $\Delta$. The Lagrangian torus fiber over the barycenter
is the Clifford torus $\T^n_{\text{cliff}} = \mu^{-1}(\bar{p})$.
Then the following holds:
\begin{theorem}[Biran--Entov--Polterovich]
   Let $q: \Ham(\C P^n, \omega_{FS}) \to \R$ be any Calabi quasimorphism which is
   continuous with respect to the Hofer metric. Then for every
   autonomous Hamiltonian $F:\C P^n\to \R$ of the form $\mu^*\bar{F}$, where $\bar{F}: \Delta\to\R$ is a smooth function, we have
   \[
   q (\phi^1_F) = \int_{\C P^n} F\omega_{FS}^n - \bar{F} (\bar{p}).
   \]
\end{theorem}
The authors then note that the result can be generalized to products of complex projective spaces, see~\cite[Section 7]{biran-entov-polterovich-2004}.
The idea is as follows: 
Let $M = \C P^{n_1} \times \ldots \times \C P^{n_k}$ be a product of complex projective spaces.
Note that if we endow $M$ with the split symplectic form $\omega_{FS}^{(1)} \oplus \ldots \oplus \omega_{FS}^{(k)}$, then $M$ is only monotone if all $n_i$ are equal.
This can be easily remedied by rescaling the symplectic form on each factor.
We set $\omega = \lambda((n_1+1)\omega_{FS}^{(1)} \oplus \ldots \oplus (n_k+1)\omega_{FS}^{(k)})$ for some $\lambda > 0$.
For simplicity, we will pick $\lambda$ such that $\vol(M) = 1$.
Note that $(M,\omega)$ is a monotone symplectic manifold and admits an Entov-Polterovich quasimorphism, see Section~\ref{sec:prelims:floer}.
Furthermore, $M$ is a toric variety with respect to the direct product of the torus actions on each factor.
Let $\Delta_i$ be the moment polytope associated with the action of $\T^{n_i}$ on the $i$-th factor $\C P^{n_i}$,
which is rescaled accordingly to account for the rescaling of the symplectic form on each factor.
The moment polytope $\Delta_M = \Delta_1 \times \ldots \times \Delta_k$ is then the moment polytope associated with the action of $\T^{n_1} \times \ldots \times \T^{n_k}$ on $M$.
We denote the corresponding moment map by $\mu_M: M \to \Delta_M$.
Let $\bar{p} = (\bar{p}_1,\ldots, \bar{p}_k) \in \Delta_M$,
where $\bar{p}_i$ denotes the barycenter of the moment polytope $\Delta_i$.
The Lagrangian torus fiber over $\bar{p}$ is the Clifford torus in the product.
By the same argument as in the case of $\C P^n$, we can see that the Entov-Polterovich quasimorphism is controlled by the values of the Hamiltonian function on this Lagrangian torus. We refer the reader to~\cite{biran-entov-polterovich-2004} for more details.
\begin{theorem}[Biran--Entov--Polterovich]
   Let $q: \Ham(M, \omega) \to \R$ be any Calabi quasimorphism which is
   continuous with respect to the Hofer metric. Then for every
   autonomous Hamiltonian $F:M\to \R$ pulled back by the moment map $\mu_M$ from a smooth function $\bar{F}: \Delta_M\to\R$,
   we have 
   \[
   q (\phi^1_F) = \int_{M} F\omega^{n_1 + \ldots + n_k} - \bar{F} (\bar{p}).
   \]
\end{theorem}

With this result in hand, we can now generalize Corollary~\ref{cor:exact-growth-S2}.
For this purpose, we say that a toric-compatible law-defining datum $\mathcal{D}$ on $\C P^{n_1} \times \ldots \times \C P^{n_k}$ is \emph{non-degenerate} if the variance of the associated Gaussian $\randH(\mu_M^{-1}(\bar{p}))$ is non-zero.
Note that since $\mathcal{D}$ is toric-compatible, any sample of the associated Gaussian process $\randH$ is invariant under the torus action, and thus it takes a single value on the Lagrangian torus fiber over $\bar{p}$. The variance in the definition above is therefore the variance of a single Gaussian variable.
It is easy to see that such a law-defining datum always exists, since we can always find an invariant eigenfunction of the Laplacian that does not vanish on the Lagrangian torus fiber over $\bar{p}$, see Remark~\ref{rmk:toric-compatible-existence}.
The same strategy as in the proof of Corollary~\ref{cor:exact-growth-S2} then allows us to obtain the following result:
\begin{corollary}\label{cor:exact-growth-CPn}
    Let $\mathcal{D}$ be a non-degenerate toric-compatible autonomous centered law-defining datum on $\C P^{n_1} \times \ldots \times \C P^{n_k}$ for $k \geq 1$ and $n_i \geq 1$ for all $i$.
    Denote by $\Phi_j$ the associated random walk on $\Ham(\C P^{n_1} \times \ldots \times \C P^{n_k},\omega)$.
    Then there exist constants $c,C > 0$ such that $c\sqrt{j} \leq \E[d_{\Hof}(\id, \Phi_j)] \leq C\sqrt{j}$ for all $j \in \Nats$.
\end{corollary}
\subsection{And beyond?}
We have now established that (whenever there exists a non-trivial Hofer-Lipschitz quasimorphism)
the random walk generated by a centered autonomously exhaustive law-defining datum $\mathcal{D}$ on a closed symplectic manifold $(M,\omega)$ has a lower bound of order $\sqrt{n}$ on the growth of $\E[d_{\Hof}(\id, \Phi_n)]$.
Thus, the random walk drifts away from the identity at least as fast as a random walk on a flat space.
In the opposite direction, we have established that if $\mathcal{D}$ is commutative, then the random walk generated by $\mathcal{D}$ has an upper bound of order $\sqrt{n}$ on the growth of $\E[d_{\Hof}(\id, \Phi_n)]$.
Thus, in this case, we can see that the random walk drifts away from the identity at most as fast as a random walk on a flat space.
If there exists a non-trivial Hofer-Lipschitz quasimorphism that does not vanish on the support of $\mathcal{D}$, then we can combine these two results to obtain that the random walk generated by $\mathcal{D}$ drifts away from the identity at exactly the same rate as a random walk on a flat space.

The obvious question that arises is: what can we say about the growth of $\E[d_{\Hof}(\id, \Phi_n)]$ in general, i.e., without restricting to commutative law-defining data? Is it still of order $\sqrt{n}$, or can it be faster?

It should be noted that $\Ham(M, \omega)$ contains many quasi-flats,
i.e., subspaces that are quasi-isometric to a flat space.
See e.g.~\cite{py-2008, usher-2013, zapolsky-2013, usher-2014, stevenson-2018, polterovich-shelukhin-2023, dawid-2024} for some examples of such quasi-flats in a variety of symplectic manifolds.
What is interesting is that in these constructions, one usually (at least implicitly) constructs a group 
homomorphism from some vector space $V$ to $\Ham(M, \omega)$, and then shows that the image of this homomorphism is a quasi-flat.
In particular, the quasi-flat is actually part of a commutative subgroup of $\Ham(M, \omega)$
and therefore in particular all of these quasi-flats are contained in $\Aut(M, \omega)$.
While they can be easily moved outside of $\Aut(M, \omega)$ by the bi-invariance of the Hofer metric, 
their geometry is still controlled by the geometry of $\Aut(M, \omega)$.
In~\cite{polterovich-shelukhin-2016}, Polterovich and Shelukhin have shown that $\Aut(M, \omega)$ is a \emph{small}
set from a metric perspective, and 
thus its geometry should not be taken as representative of the geometry of $\Ham(M, \omega)$
as a whole.
Thus, even the infinite-dimensional quasi-flats mentioned above might not mean that $\Ham(M, \omega)$ behaves like a flat space in general.
In Theorem~\ref{thm:upper-bound-commutative} we have seen that 
random walks constrained to a commutative subgroup of $\Ham(M, \omega)$ behave 
as expected on flat spaces.
This result is not entirely unexpected, since the construction of quasi-flats in $\Ham(M, \omega)$ usually passes through a commutative subgroup of $\Ham(M, \omega)$.
However, the question remains: do random walks on the whole group $\Ham(M, \omega)$ also behave like random walks on flat spaces?

We will now give a brief discussion of why this question is far more challenging than it might seem at first glance, and why the techniques we have used so far do not seem to be sufficient to answer it.
Of course, the way we would go about answering the question depends heavily on what the answer is.
Let us first lay out how one might prove an upper bound of order $\sqrt{n}$ on the growth of $\E[d_{\Hof}(\id, \Phi_n)]$ for a general law-defining datum $\mathcal{D}$.
For this it would suffice to produce an explicit random Hamiltonian function
with the correct law which has a $\infty$-norm growth of order $\sqrt{n}$.
One possible approach is as follows:
Let $\mathcal{D}$ be a centered autonomously exhaustive law-defining datum on a closed symplectic manifold $(M,\omega)$, and let $\Phi_n$ be the associated random walk on $\Ham(M,\omega)$.
Let $\randH_{1}, \dots, \randH_{n}$ be independent versions of the process~\eqref{eq:GP_H} associated with $\mathcal{D}$
and let $\{\Phi_n\}_{n \in \Nats}$ be the associated random walk on $\Ham(M,\omega)$.
Then we have that $\Phi_n$ has the same law as $\phi_{\randH_n}^1 \circ \cdots \circ \phi_{\randH_1}^1$.
Since $\mathcal{D}$ induces an inversion invariant measure, this is the same law as that of $\phi_{\randH_1}^1 \circ \cdots \circ \phi_{\randH_n}^1$.
The latter is the time-$1$ map of the time-dependent Hamiltonian function $\randH_1 \# \cdots \# \randH_n$
given by 
\begin{align*}
    \tilde{H}^{\randomness}_n (t,x)\coloneqq\randH_{1} \# \cdots \# \randH_{n} (t,x) & =
    \sum_{i=1}^{n}
    \left(
    \randH_i \circ
    \phi_{\randH_{i-1}}^{-t} \circ \cdots \circ \phi_{\randH_1}^{-t}
    \right)(x)
    \\
    &=
    \sum_{i=1}^{n}
    \randH_i(
    {(\phi_{\randH_{1} \# \cdots \# \randH_{i-1}}^{t})}^{-1}(x)).
\end{align*}
We now wish to show that $\norm{\tilde{H}^{\randomness}_n}_{\infty}$ grows at most like $\sqrt{n}$.
One approach would be to try to expand $\tilde{H}^{\randomness}_n$ in the eigenbasis of the Laplacian given by $\mathcal{D}$, and then try to control the growth of the coefficients of this expansion.
However, to do this we would need to control the coefficients of $e_k \circ {(\phi_{\randH_{1} \# \cdots \# \randH_{i-1}}^{t})}^{-1}$ for all $i$ and $k$.
A priori, it seems that controlling how far these coefficients spread is impossible in any but the most rudimentary way.
A second approach would be to use general methods for the estimation of the maximum of a stochastic process, and apply them to the process $\tilde{H}^{\randomness}_n$ directly.
The problem with this approach is that methods such as Dudley's entropy bound or 
Talagrand's majorization theorem require the process to be (sub-)Gaussian, which is not the case for $\tilde{H}^{\randomness}_n$.
This is due to the fact that the distortion by the flow of the previous Hamiltonians $\phi_{\randH_{1} \# \cdots \# \randH_{i-1}}^{t}$ acts on the covariance structure of the process in a complex way, and thus the resulting process is not Gaussian, but rather a mixture of Gaussian processes.
The author's approach to try to adapt the techniques of Dudley's entropy bound to this setting has not been successful at yielding a bound better than the trivial linear bound.
While this by no means rules out the possibility of obtaining a better bound, there seems to not be a clear approach to do so at this time.

\begin{figure}[h]
    \centering
    \includegraphics[width=0.48\textwidth]{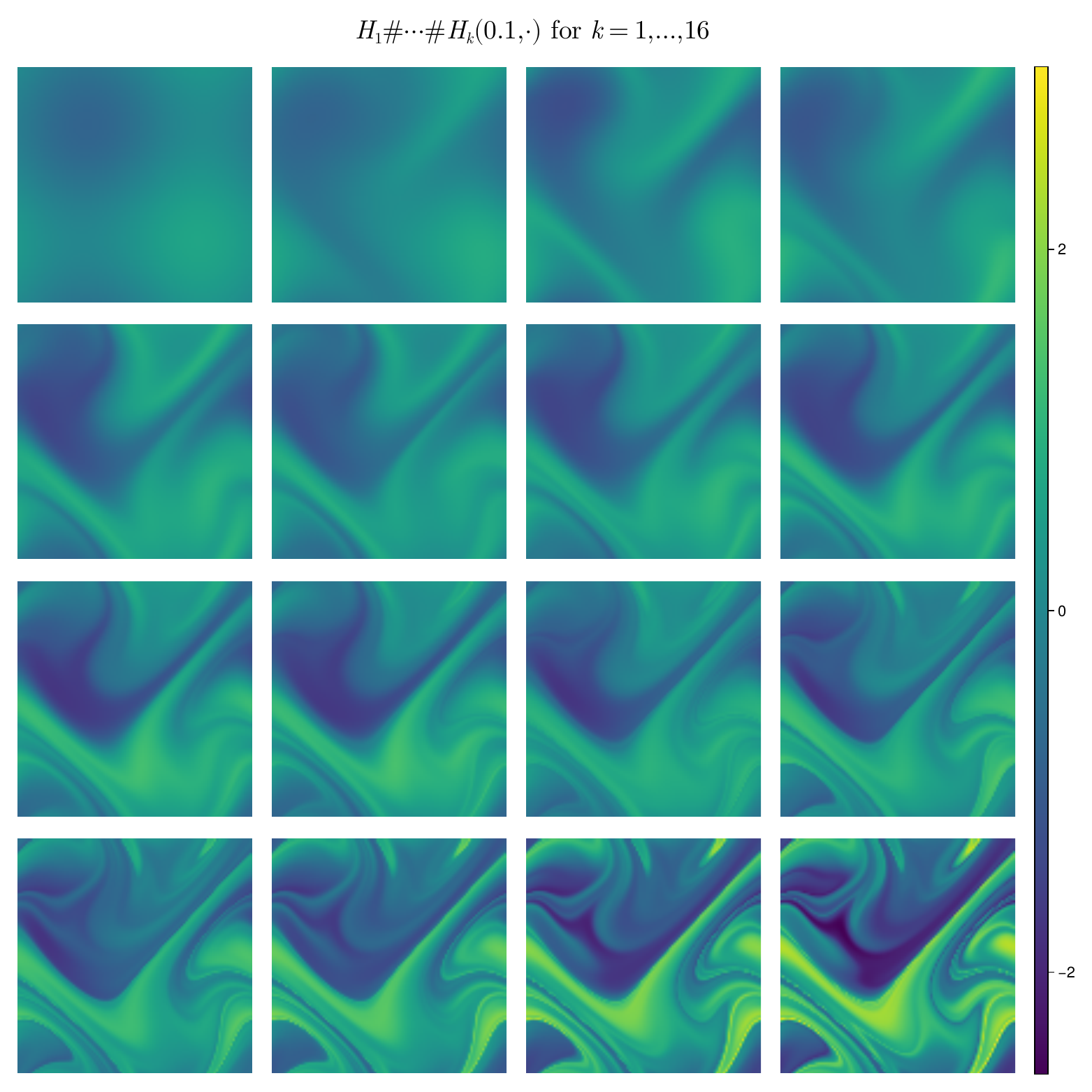}
    \includegraphics[width=0.48\textwidth]{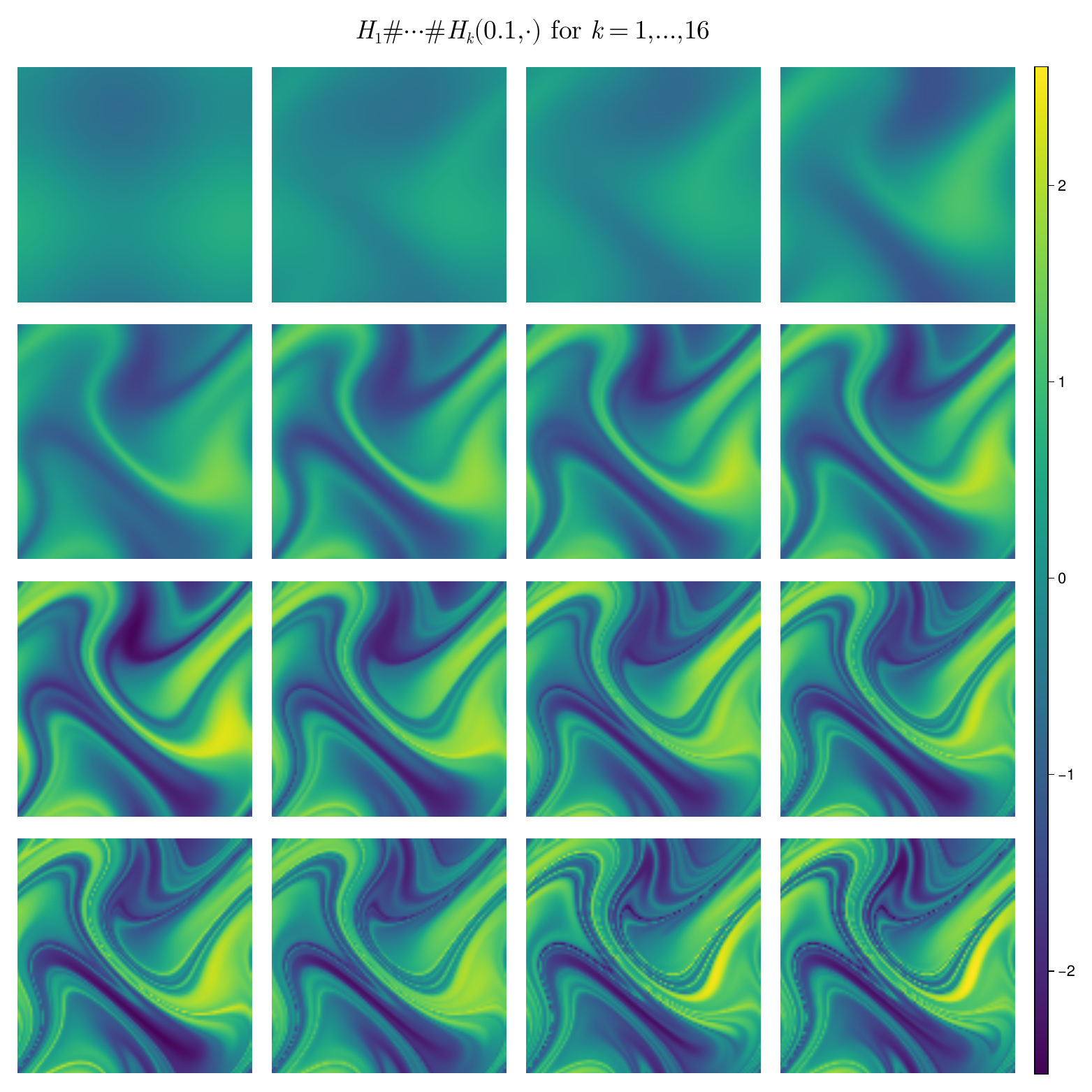}
    \caption{The evolution of the Hamiltonian function under the first 16 steps of the random walk for two draws at time $t=0.1$.}
    \label{fig:H-under-random-walk}
\end{figure}
This might lead us to speculate that the growth of $\E[d_{\Hof}(\id, \Phi_n)]$ might be faster than $\sqrt{n}$ in general.
If so, instead of an improved upper bound, we would need an improved lower bound on the growth of $\E[d_{\Hof}(\id, \Phi_n)]$.
This also seems to be a difficult problem.
The central limit theorem for quasimorphisms (Theorem~\ref{thm:q-clt}) cannot be applied in this case, since it always results in a lower bound of order $\sqrt{n}$.
It would thus be necessary to have a similar statement about the limiting distribution for another known lower bound of the Hofer norm, such 
as the boundary depth or some other quantitative invariant.

One small way in which we can get insight is to estimate the growth of $\tilde{H}^{\randomness}_n$ numerically. While this is not a true proxy for the Hofer norm, which could be vastly smaller, it can give us some insight into the behavior of $\tilde{H}^{\randomness}_n$.
Figure~\ref{fig:H-under-random-walk} shows a visualization of the evolution of the Hamiltonian function under the first 16 steps of the random walk for two draws at time $t=0.1$.
Note that in this non-commutative setting, the Hamiltonian function $\tilde{H}^{\randomness}_n$ is truly time-dependent.

The first thing to note is that the Hamiltonian function seems to become much less smooth over time, presenting a problem for any attempt to control the coefficients of the eigenfunction expansion.
The second thing to note is that the Hamiltonian function seems to form certain persistent patterns over time, and that the overall oscillation of the Hamiltonian function
therefore seems to grow faster than if we were to simply add up independent Gaussian processes,
since fewer cancellations occur.
The numerical experiments seem to confirm this, as can be seen in the figure below (Figure~\ref{fig:H-osc}), which shows the growth of
the expected value of $\osc(\tilde{H}^{\randomness}_n(0.1, \cdot))$ over the first 20 steps of the random walk for 600 independent draws.
\begin{figure}[h!]
    \centering
    \includegraphics[width=\textwidth]{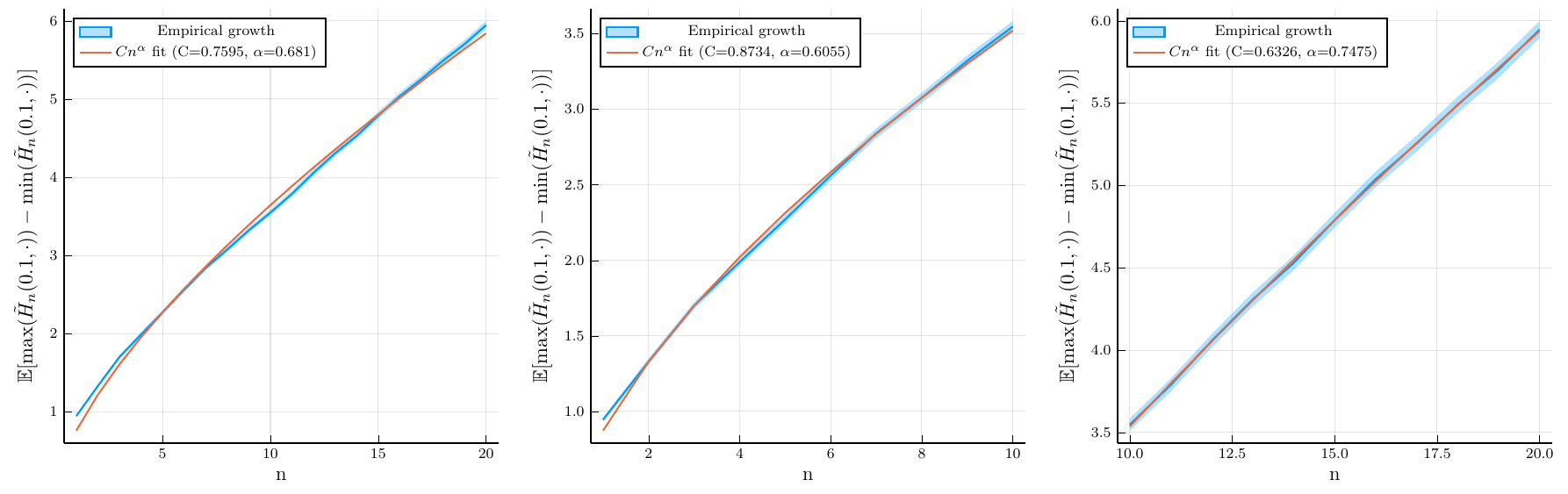}
    \caption{The growth of the expected value of the oscillation of the Hamiltonian function $\tilde{H}^{\randomness}_n$ at time $t=0.1$ after $n$ steps of the random walk. The growth is estimated from 600 independent draws.}\label{fig:H-osc}
\end{figure}
When we study this figure, we can see that the growth of the expected value of $\osc(\tilde{H}^{\randomness}_n(0.1, \cdot))$ seems to be faster than $\sqrt{n}$, and might even be linear in $n$. The figure contains curves of the form $C \cdot n^{\alpha}$ fitted to the data.
We can see that for the first few steps of the random walk, the growth is clearly sublinear, although it already seems to be faster than $\sqrt{n}$.
This makes some sense: given any two autonomous Hamiltonians $H_1$ and $H_2$, the oscillation of $H_1 \# H_2$ is exactly the same as the oscillation of $H_1 + H_2$. This phenomenon breaks down when we start concatenating more (non-commuting) Hamiltonians, and the oscillation of $H_1 \# H_2 \# H_3$ can be larger than the oscillation of $H_1 + H_2 + H_3$. Thus, the growth of the oscillation of $\tilde{H}^{\randomness}_n$ is slower in the beginning.
As we can see from the figure, the growth accelerates when we discard the first few steps. It seems entirely believable that we approach 
linear growth asymptotically, although we cannot make any rigorous statements about this at this time.
If one could rigorously show a similar growth rate for the Hofer norm of $\Phi_n$, this would imply that the random walk generated by a general law-defining datum $\mathcal{D}$ drifts away from the identity faster than a random walk on a flat space. This would reveal the presence of hyperbolic geometry in $\Ham(M, \omega)$ that has so far stayed undetected.

\bibliographystyle{alpha}
\bibliography{refs}

@article{branson-2011,
	author = {Branson, Mark},
	doi = {10.2140/agt.2011.11.1077},
	fjournal = {Algebraic \& Geometric Topology},
	issn = {1472-2747,1472-2739},
	journal = {Algebr. Geom. Topol.},
	mrclass = {53D35 (20F69 53D45)},
	mrnumber = {2792374},
	mrreviewer = {Vincent\ Humili\`ere},
	number = {2},
	pages = {1077--1096},
	title = {Symplectic manifolds with vanishing action-{M}aslov homomorphism},
	url = {https://doi.org/10.2140/agt.2011.11.1077},
	volume = {11},
	year = {2011}}

@article{gambaudo-ghys-2004,
	author = {Gambaudo, Jean-Marc and Ghys, \'Etienne},
	doi = {10.1017/S0143385703000737},
	fjournal = {Ergodic Theory and Dynamical Systems},
	issn = {0143-3857,1469-4417},
	journal = {Ergodic Theory Dynam. Systems},
	mrclass = {37E30 (20F28 37C05 57M60)},
	mrnumber = {2104597},
	mrreviewer = {Martin\ J.\ Schmoll},
	number = {5},
	pages = {1591--1617},
	title = {Commutators and diffeomorphisms of surfaces},
	url = {https://doi.org/10.1017/S0143385703000737},
	volume = {24},
	year = {2004}}

@article{brandenbursky-marcinkowski-shelukhin-2022,
	author = {Brandenbursky, Michael and Marcinkowski, Micha\l{} and Shelukhin, Egor},
	fjournal = {Selecta Mathematica. New Series},
	journal = {Selecta Math. (N.S.)},
	number = {4},
	pages = {Paper No. 74, 20},
	title = {The {S}chwarz-{M}ilnor lemma for braids and area-preserving diffeomorphisms},
	volume = {28},
	year = {2022}}

@article{shnirelman-1994,
	author = {Shnirelman, A. I.},
	doi = {10.1007/BF01896409},
	fjournal = {Geometric and Functional Analysis},
	journal = {Geom. Funct. Anal.},
	number = {5},
	pages = {586--620},
	title = {Generalized fluid flows, their approximation and applications},
	volume = {4},
	year = {1994}}

@article{cheuk-yu-trifa-2026,
	author = {Mak, Cheuk Yu and Trifa, Ibrahim},
	doi = {10.4171/cmh/601},
	fjournal = {Commentarii Mathematici Helvetici. A Journal of the Swiss Mathematical Society},
	issn = {0010-2571,1420-8946},
	journal = {Comment. Math. Helv.},
	mrclass = {53D40 (37E30 37J06)},
	mrnumber = {5042352},
	number = {2},
	pages = {345--392},
	title = {Hameomorphism groups of positive genus surfaces},
	url = {https://doi.org/10.4171/cmh/601},
	volume = {101},
	year = {2026}}

@article{milnor-1976,
	author = {John Milnor},
	doi = {https://doi.org/10.1016/S0001-8708(76)80002-3},
	issn = {0001-8708},
	journal = {Advances in Mathematics},
	number = {3},
	pages = {293-329},
	title = {Curvatures of left invariant metrics on lie groups},
	url = {https://www.sciencedirect.com/science/article/pii/S0001870876800023},
	volume = {21},
	year = {1976}}

@article{polterovich-shelukhin-2016,
	author = {Polterovich, Leonid and Shelukhin, Egor},
	doi = {10.1007/s00029-015-0201-2},
	fjournal = {Selecta Mathematica. New Series},
	issn = {1022-1824,1420-9020},
	journal = {Selecta Math. (N.S.)},
	mrclass = {53D05 (37K05)},
	mrnumber = {3437837},
	mrreviewer = {Karl\ Friedrich\ Siburg},
	number = {1},
	pages = {227--296},
	title = {Autonomous {H}amiltonian flows, {H}ofer's geometry and persistence modules},
	url = {https://doi.org/10.1007/s00029-015-0201-2},
	volume = {22},
	year = {2016}}

@article{schwarz-1975,
	author = {Schwarz, Gerald W.},
	doi = {10.1016/0040-9383(75)90036-1},
	fjournal = {Topology. An International Journal of Mathematics},
	issn = {0040-9383},
	journal = {Topology},
	mrclass = {58C25 (57E15)},
	mrnumber = {370643},
	mrreviewer = {G.\ R.\ Belitski\u i},
	pages = {63--68},
	title = {Smooth functions invariant under the action of a compact {L}ie group},
	url = {https://doi.org/10.1016/0040-9383(75)90036-1},
	volume = {14},
	year = {1975}}

@article{guillemin-1994,
	author = {Guillemin, Victor},
	fjournal = {Journal of Differential Geometry},
	issn = {0022-040X,1945-743X},
	journal = {J. Differential Geom.},
	number = {2},
	pages = {285--309},
	title = {Kaehler structures on toric varieties},
	volume = {40},
	year = {1994}}

@article{watson-1973,
	author = {Watson, Bill},
	doi = {10.4310/jdg/1214431482},
	issn = {0022-040X},
	journal = {Journal of Differential Geometry},
	month = jan,
	number = {1},
	title = {Manifold maps commuting with the {Laplacian}},
	volume = {8},
	year = {1973}}

@article{cheng-1975,
	author = {Cheng, Shiu-Yuen},
	doi = {10.1007/BF01214381},
	issn = {1432-1823},
	journal = {Mathematische Zeitschrift},
	month = oct,
	number = {3},
	pages = {289--297},
	title = {Eigenvalue comparison theorems and its geometric applications},
	volume = {143},
	year = {1975}}

@article{donnelly-2001,
	author = {Donnelly, Harold},
	doi = {10.1006/jfan.2001.3817},
	issn = {0022-1236},
	journal = {Journal of Functional Analysis},
	month = dec,
	number = {1},
	pages = {247--261},
	title = {Bounds for {Eigenfunctions} of the {Laplacian} on {Compact} {Riemannian} {Manifolds}},
	volume = {187},
	year = {2001}}

@article{bavard-1991,
	author = {Bavard, Christophe},
	fjournal = {L'Enseignement Math\'ematique. Revue Internationale. 2e S\'erie},
	issn = {0013-8584},
	journal = {Enseign. Math. (2)},
	mrclass = {20F12 (20J05 57M07)},
	mrnumber = {1115747},
	mrreviewer = {Darryl\ McCullough},
	number = {1-2},
	pages = {109--150},
	title = {Longueur stable des commutateurs},
	volume = {37},
	year = {1991}}

@article{zapolsky-2013,
	author = {Zapolsky, Frol},
	journal = {Journal of Symplectic Geometry},
	month = aug,
	number = {3},
	pages = {475--488},
	title = {On the {Hofer} geometry for weakly exact {Lagrangian} submanifolds},
	volume = {11},
	year = {2013}}

@article{usher-2014,
	author = {Michael Usher},
	journal = {Journal of Symplectic Geometry},
	number = {3},
	pages = {619 -- 656},
	publisher = {International Press of Boston},
	title = {{Hofer Geometry and cotangent fibers}},
	volume = {12},
	year = {2014}}

@article{py-2008,
	author = {Py, Pierre},
	doi = {10.1515/CRELLE.2008.053},
	journal = {Journal f{\"u}r die reine und angewandte Mathematik},
	pages = {185--193},
	title = {Quelques plats pour la m\'etrique de {Hofer}},
	volume = {620},
	year = {2008}}

@article{usher-2013,
	author = {Usher, Michael},
	doi = {10.24033/asens.2185},
	journal = {Annales scientifiques de l'{\'E}cole Normale Sup{\'e}rieure},
	number = {1},
	pages = {57--129},
	series = {4},
	title = {{Hofer}'s metrics and boundary depth},
	volume = {46},
	year = {2013}}

@article{stevenson-2018,
	author = {Stevenson, Bret},
	journal = {Israel Journal of Mathematics},
	number = {1},
	pages = {141--195},
	title = {A quasi-isometric embedding into the group of {Hamiltonian} diffeomorphisms with {Hofer}'s metric},
	volume = {223},
	year = {2018}}

@article{polterovich-shelukhin-2023,
	author = {Polterovich, Leonid and Shelukhin, Egor},
	doi = {10.1112/S0010437X23007455},
	journal = {Compositio Mathematica},
	number = {12},
	pages = {2483--2520},
	title = {Lagrangian configurations and {Hamiltonian} maps},
	volume = {159},
	year = {2023}}

@article{dawid-2024,
	author = {Dawid, Adrian},
	issn = {1540-2347},
	journal = {Journal of Symplectic Geometry},
	number = {6},
	pages = {1235--1286},
	publisher = {International Press of Boston},
	title = {Hofer geometry of ${A}_3$-configurations},
	volume = {23},
	year = {2025}}

@misc{edtmair-2025,
	archiveprefix = {arXiv},
	author = {Oliver Edtmair},
	eprint = {2509.16327},
	primaryclass = {math.SG},
	title = {Smooth perfectness of Hamiltonian diffeomorphism groups},
	url = {https://arxiv.org/abs/2509.16327},
	year = {2025}}

@article{fukaya-oh-ohta-ono-2019,
	author = {Fukaya, Kenji and Oh, Yong-Geun and Ohta, Hiroshi and Ono, Kaoru},
	doi = {10.1090/memo/1254},
	journal = {Memoirs of the American Mathematical Society},
	number = {1254},
	title = {Spectral invariants with bulk, quasimorphisms and {L}agrangian {F}loer theory},
	volume = {260},
	year = {2019}}

@article{usher-2011,
	author = {Usher, Michael},
	doi = {10.2140/gt.2011.15.1313},
	journal = {Geometry \& Topology},
	number = {3},
	pages = {1313--1417},
	title = {Deformed {H}amiltonian {F}loer theory, capacity estimates, and {C}alabi quasimorphisms},
	volume = {15},
	year = {2011}}

@article{biran-entov-polterovich-2004,
	author = {Biran, Paul and Entov, Michael and Polterovich, Leonid},
	doi = {10.1142/S0219199704001525},
	journal = {Communications in Contemporary Mathematics},
	number = {5},
	pages = {793--802},
	title = {Calabi quasimorphisms for the symplectic ball},
	volume = {6},
	year = {2004}}

@article{entov-polterovich-2003,
	author = {Entov, Michael and Polterovich, Leonid},
	doi = {10.1155/S1073792803210011},
	journal = {International Mathematics Research Notices},
	number = {30},
	pages = {1635--1676},
	title = {Calabi quasimorphism and quantum homology},
	volume = {2003},
	year = {2003}}

@article{stout-1970,
	author = {William F. Stout},
	issn = {00034851, 21688990},
	journal = {The Annals of Mathematical Statistics},
	number = {6},
	pages = {2158--2160},
	publisher = {Institute of Mathematical Statistics},
	title = {The Hartman-Wintner Law of the Iterated Logarithm for Martingales},
	url = {http://www.jstor.org/stable/2240358},
	urldate = {2026-02-18},
	volume = {41},
	year = {1970}}

@article{billingsley-1961,
	author = {Patrick Billingsley},
	issn = {00029939, 10886826},
	journal = {Proceedings of the American Mathematical Society},
	number = {5},
	pages = {788--792},
	publisher = {American Mathematical Society},
	title = {The Lindeberg-L{\'e}vy Theorem for Martingales},
	url = {http://www.jstor.org/stable/2034876},
	urldate = {2026-02-10},
	volume = {12},
	year = {1961}}

@article{bjorklund-hartnick-2011,
	author = {Bj{\"o}rklund, Michael and Hartnick, Tobias},
	doi = {10.2140/gt.2011.15.123},
	issn = {1465-3060},
	journal = {Geometry \& Topology},
	month = jan,
	number = {1},
	pages = {123--143},
	publisher = {Mathematical Sciences Publishers},
	title = {Biharmonic functions on groups and limit theorems for quasimorphisms along random walks},
	volume = {15},
	year = {2011}}

@article{dawid-2025,
	archiveprefix = {arXiv},
	author = {Adrian Dawid},
	eprint = {2510.03190},
	note = {\textit{arXiv:2510.03190}},
	primaryclass = {math.SG},
	title = {{R}andom {H}amiltonians {I}: {P}robability measures and random walks on the {H}amiltonian diffeomorphism group},
	url = {https://arxiv.org/abs/2510.03190},
	year = {2025}}

@article{cristofaro-gardiner-humiliere-mak-seyfaddini-smith-2022,
	author = {Cristofaro-Gardiner, Daniel and Humili{\`e}re, Vincent and Mak, Cheuk Yu and Seyfaddini, Sobhan and Smith, Ivan},
	doi = {10.1017/fmp.2022.18},
	journal = {Forum of Mathematics, Pi},
	pages = {e27},
	title = {Quantitative Heegaard Floer cohomology and the Calabi invariant},
	volume = {10},
	year = {2022}}

@book{villani-2009,
	author = {Villani, C\'edric},
	doi = {10.1007/978-3-540-71050-9},
	isbn = {978-3-540-71049-3},
	mrclass = {49-02 (28A75 37J50 49Q20 53C23 58E30)},
	mrnumber = {2459454},
	mrreviewer = {Dario\ Cordero-Erausquin},
	note = {Old and new},
	pages = {xxii+973},
	publisher = {Springer-Verlag, Berlin},
	series = {Grundlehren der mathematischen Wissenschaften [Fundamental Principles of Mathematical Sciences]},
	title = {Optimal transport},
	url = {https://doi.org/10.1007/978-3-540-71050-9},
	volume = {338},
	year = {2009}}

@book{mcduff-salamon-2017,
	author = {McDuff, Dusa and Salamon, Dietmar},
	doi = {10.1093/oso/9780198794899.001.0001},
	edition = {Third},
	isbn = {978-0-19-879490-5; 978-0-19-879489-9},
	mrclass = {53D35 (53D40 57R17 57R57 57R58)},
	mrnumber = {3674984},
	mrreviewer = {Hansj\"org\ Geiges},
	pages = {xi+623},
	publisher = {Oxford University Press, Oxford},
	series = {Oxford Graduate Texts in Mathematics},
	title = {Introduction to symplectic topology},
	url = {https://doi.org/10.1093/oso/9780198794899.001.0001},
	year = {2017}}

\end{document}